\documentclass[twocolumn]{autart}
\usepackage{color}
\usepackage{tikz} 
\usepackage{float}
\usepackage{bm}
\usepackage{amsmath}
\usepackage{amssymb,amsfonts}
\usepackage{epstopdf}
\usepackage{graphicx}      
\usepackage{mathtools}
\usepackage{subcaption}
\usepackage[export]{adjustbox}
\usepackage{picins}

\AtBeginDocument{
  \setlength{\abovedisplayskip}{4pt plus 2pt minus 1pt}%
  \setlength{\belowdisplayskip}{4pt plus 2pt minus 1pt}%
  \setlength{\abovedisplayshortskip}{4pt plus 2pt minus 1pt}%
  \setlength{\belowdisplayshortskip}{4pt plus 2pt minus 1pt}%
}
\def\dref#1{(\ref{#1})}
\renewcommand{\cline}{{\mathbb C}}
\newcommand{\fline}  {{\mathbb F}}
\newcommand{\nline}  {{\mathbb N}}
\newcommand{\rline}  {{\mathbb R}}
\newcommand{\tline}  {{\mathbb T}}

\newcommand{\qq}   {{\bf q}}
\newcommand{\uu}   {{\bf u}}
\newcommand{\xx}   {{\bf x}}
\newcommand{\yy}   {{\bf y}}
\newcommand{\GGG}  {{\mathbf G}}
\renewcommand{\gg} {{\mathbf g}}
\newcommand{\HHH}  {{\mathrm H}}

\newcommand{\PPP}  {{\mathbf P}}

\newcommand{\Sig}  {{\mathbf \Sigma}}

\newcommand{\dd}   {{\rm d}\hbox{\hskip 0.5pt}}
\renewcommand{\leq} {\leqslant}
\renewcommand{\geq} {\geqslant}
\newcommand{\Ran} {{\mathrm {Ran}}}
\newcommand{\Ker} {{\mathrm {Ker}}}
\newcommand{\Ascr} {{\mathcal A}}

\newcommand{\Dscr} {{\mathcal D}}

\newcommand{\Gscr} {{\mathcal G}}
\newcommand{\Hscr} {{\mathcal H}}
\newcommand{\Jscr} {{\mathcal J}}

\newcommand{\Lscr} {{\mathcal L}}

\newcommand{\Nscr} {{\mathcal N}}

\newcommand{\mm}    {{\hbox{\hskip 0.5pt}}}
\newcommand{\m}     {{\hbox{\hskip 1pt}}}
\newcommand{\n}     {{\hbox{\hskip -5pt}}}
\newcommand{\nm}    {{\hbox{\hskip -3pt}}}

\newcommand{\bluff} {{\hbox{\raise 15pt \hbox{\mm}}}}
\newcommand{\sbluff}{{\hbox{\raise 10pt \hbox{\mm}}}}
\newcommand{\Om}    {{\Omega}}

\renewcommand{\d}   {{\rm d}}

\renewcommand{\o}   {{\omega}}
\renewcommand{\e}   {{\varepsilon}}
\newcommand{\FORALL}{{\hbox{$\hskip 10mm \forall \;$}}}
\newcommand{\rarrow}{\mathop{\rightarrow}}

\newcommand{\LE}[1] {{L^2([0,\infty);#1)}}
\newcommand{\LEloc}[1] {{L^2_{\rm loc}([0,\infty);#1)}}

\newcommand{\CD}     {{C\& D}}

\newcommand{\Dom}[1]{\Dscr(#1)}

\newcommand{\loc}{{\rm loc}}
\renewcommand{\Re}{{\rm Re\,}}
\renewcommand{\Im}{{\rm Im\,}}
\newcommand{\bbm}[1]{\left[\begin{matrix} #1 \end{matrix}\right]}
\newcommand{\sbm}[1]{\left[\begin{smallmatrix} #1
\end{smallmatrix}\right]}

\newcommand{\rfb}[1]{\mbox{\rm
(\ref{#1})}\ifx\undefined\stillediting\else:\fbox{$#1$}\fi}
\renewcommand{\theenumi}{\roman{enumi}}

\makeatletter
\renewcommand{\p@enumii}{}
\makeatother
\definecolor{forestgreen}{rgb}{0.13, 0.55, 0.13}
\def\dfrac{\displaystyle\frac}
\renewcommand{\theequation}{\arabic{section}.\arabic{equation}}
\newtheorem{theorem}{Theorem}[section]
\newtheorem{definition}{Definition}[section]
\newtheorem{proposition}{Proposition}[section]
\newtheorem{lemma}{Lemma}[section]
\newtheorem{remark}{Remark}[section]
\newtheorem{example}{Example}[section]
\newtheorem{assumption}{Assumption}[section]

\date{}
\begin{document}
\begin{frontmatter}
\title{The \m linear \m regulator \m problem \m for \m passive \m
       systems\\ with \m strong \m stability\thanksref{footnoteinfo}
       \vspace{-3mm}}
\thanks[footnoteinfo]{This work was supported by the Israel
Science Foundation-NSFC grant no. 3621/21,
by the National Natural Science Foundation of China (no. 62173348,
12161141013) and Hunan Basic Science Research Center for Mathematical
Analysis (2024JC2002).}
\author[First,Second]{Bingsen Li}\ead{bingsenli9@gmail.com}
\author[Second]{Hua-Cheng Zhou}\ead{hczhou@amss.ac.cn}
\author[First]{George Weiss}\ead{gweiss@tauex.tau.ac.il}
\address[First]{School of Electrical and Computer Engineering, Tel
  Aviv University, Ramat Aviv 69978, Israel }
\address[Second]{School of Mathematics and Statistics, HNP-LAMA,
  Central South University, Changsha  410083, P.R. China }
%
\begin{abstract}
We study the output regulator problem for an impedance passive linear
plant, using the classical resonant internal model based controller.
The reference and disturbance signals are assumed to be linear
combinations of sine waves of known frequencies. {We prove that,
under mild assumptions and without requiring the stability of the
plant, this controller renders the closed-loop system strongly stable.
Furthermore, it solves the output regulator problem in the sense that
the tracking error tends to zero when filtered -- meaning that passing
the error through a first-order low-pass filter yields a continuous
signal that converges to zero.} We give two examples (both models of
engineering systems) to illustrate the results. \vspace{-2mm}
\end{abstract}

\begin{keyword}
   Output regulator problem, passive system, strong stability,
   disturbance rejection, internal model principle
\end{keyword}

\end{frontmatter}

\section{Introduction} \label{sec1} 
\setcounter{equation}{0}

The \textit{output regulator problem} (also called the regulator
problem) has attracted extensive attention and research in control
theory. In this problem, the reference and disturbance signals
acting on the plant are generated by a marginally stable exosystem,
and the aim is to design a stabilizing feedback controller such that
the reference signal is tracked asymptotically by the system output.
In the 1970's, the robust output regulator problem for finite
dimensional linear systems was studied, based on the newly discovered
\textit{internal model principle}, by \cite{FraWon} and
\cite{MR0406616}. The internal model principle, which characterizes
the controllers that achieve robust output regulation, is a primary
tool in control theory. It tells us that, under suitable assumptions,
the regulator problem can be solved by including an appropriate model
of the exosystem in the controller and by selecting the controller's
remaining parameters in a way that ensures the stability of the
closed-loop system composed of the plant and the controller. The
principle has been extended to nonlinear systems, see for instance,
\cite{MR1015932}, \cite{knobloch2012topics}, \cite{byrnes1997output}
and \cite{astolfi2003immersion}, for a brief overview see also Sect.
2-3 of \cite{weiss2019minimal}.

In the last 25 years, many results have appeared about the output
regulation of \textit{distributed parameter systems} (DPS).
\cite{MR1807306} considers the output regulator problem for DPS
with a finite{-}dimensional exosystem, for a plant with bounded
control and observation operators, while \cite{MR2285715},
\cite{MR2735508} give results about the output regulation of DPS with
an infinite{-}dimensional exosystem. In \cite{MR3285896}, the
internal model principle is generalized to DPS with unbounded control
and observation operators. However, they still require that certain
operators in the dynamic error feedback system are bounded.
\cite{Paun2017rob} (part of a series of papers) investigates the
unbounded operators situation, but then the convergence of the
tracking error is more difficult to formulate, because the error may
not be defined pointwise (we will return to this point later).
\cite{art59} proposes a low gain controller based on the internal
model principle, for tracking and disturbance rejection for a stable
well-posed plant with a finite{-}dimensional exosystem, and
identifies a simple (not low gain based) controller when a certain
transfer function is positive. \cite{art66} has solved the state
feedback regulator problem for regular linear systems. {For
specific DPS, time-domain techniques such as backstepping
\cite{MR3537141}, \cite{art125}, feedforward \cite{GuoMeng}, and
Lyapunov-based methods \cite{Andrieu} also provide effective control
tools.}

For the case of unknown frequencies, an adaptive method has been
developed in \cite{Marino2003TAC} and \cite{Marino2013Auto} for
systems described by ordinary differential equations and then
adapted in \cite{GuoZhao2022Adptive} for PDE systems.

An important step in solving the regulator problem is to make sure
that the closed-loop system is stable. This is often done by adding an
observer-based controller to the system formed by the plant and the
internal model. Sometimes we are in the fortunate situation that the
coupled system formed from the plant and the internal model is stable
if we choose the parameters of the internal model correctly, and then
the solution of the regulation problem becomes much more simple and
elegant {(but the performance may be limited)}. This has been
explored in \cite{davison1976multivariable}, and for
infinite-dimensional systems in \cite{Hamalainen2000},
\cite{Lassi2019}, \cite{art59} and \cite{art81}. This paper is a
contribution to this line of investigation, where we consider a
possibly unstable infinite-dimensional plant.

The stability properties (including the strong stability) of the
feedback interconnection of two impedance passive linear
time-invariant systems, one of which is finite-dimensional, have been
investigated in \cite{art114}. The finite-dimensional subsystem is
assumed to be input-output stable, so that it cannot be an
\textit{internal model controller} (\textbf{IMC}). We refer to
\cite{art114} and \cite{Lassi2025JFA} for other references on the
stability of coupled infinite-dimensional systems. The conference
paper \cite{art136} shows the strong stability of a similar coupled
system, but now the finite-dimensional subsystem is allowed to have
imaginary eigenvalues, as long as they do not coincide with the zeros
or the eigenvalues of the DPS. We use results from \cite{art136},
while also providing the proof of some key statements that were not
proved there.

We consider coupled systems consisting of a well-posed and impedance
passive linear system $\Sig_p$ (the plant) and an \textbf{IMC},
connected in feedback as shown in Fig.~\ref{fig:blocks}. The {\em
disturbance signals} $w$ and $v$, and the {\em reference signal} $r$
are supposed to have the following form:
\begin{subequations} \label{exosig}
   \begin{align} & w(t) \m=\m \Sigma_{j\in\Jscr} w_j e^{i\o_j t},
   \quad v(t) \m=\m \Sigma_{j\in\Jscr} v_j e^{i\o_j t},
   \label{d1}\m\\ & r(t) \m=\m \Sigma_{j\in\Jscr} r_j e^{i\o_j t},
   \label{d2} \end{align}
\end{subequations} \vspace{1mm} \noindent \m \hspace{-2.5mm}
where $\Jscr=\{1,2,\ldots n\}$. The above functions can represent a
combination of step functions of arbitrary magnitudes and sinusoidal
functions of arbitrary amplitudes and initial phases, which is
commonly encountered in engineering. If these signals are real, then
the set of frequencies $\Om=\{\o_j\ |\ j\in\Jscr\}$ must be
symmetric: $-\Om=\Om$.

We restrict the number of the frequencies of disturbance and reference
signals to be finite. This limitation is necessary to prevent the
following robustness issue: if we employ an error feedback controller
that drives the {\em tracking error} \vspace{-2mm}
\begin{equation} \label{trac_er}
   e(t) \m=\m y(t)-r(t)
\end{equation} \vspace{1.5mm} \noindent \m \hspace{-2.5mm}
to zero, then all the points $i\o_j$ must be poles of the controller,
following the internal model principle. As noted by \cite[Theorem 1.2]
{logemann1996conditions}, the closed-loop system cannot be robustly
stable with respect to small delays in the feedback loop when the
plant or the controller has an infinite number of unstable poles.

We assume that $\Sig_p$ is a well-posed linear system with state space
$X$, with inputs $w$ and $u$, and output $y$. {The property of
well-posedness ensures that on any finite time interval, inputs in
$L^2$ produce continuous state trajectories and outputs in $L^2$. For
the background on well-posed systems, we refer to
\cite{staffans2005well}, \cite{art11}, \cite{art12}, \cite{art10}.}
The disturbance signal $w$ and the control input $u$ take values in
the Hilbert spaces $W$ and $U=\cline^p$, respectively. The output
signal $y$ of $\Sig_p$ and the reference signal $r$ also take values
in $U$. The transfer function of $\Sig_p$ from $[w\ \ u]^\top$ to $y$
is denoted by \m $\PPP=[\PPP_w\ \PPP_u]$.

We assume that $\Sig_p$ is {\em impedance passive} from $u$ to $y$,
which means that if $w=0$, then for any initial state $x(0)\in X$, any
$\tau>0$ and any $u\in L^2([0,\tau];U)$, {the state trajectory
$x$ satisfies (see for instance \cite{staffans2002passive,%
staffans2005well,art69})
\begin{equation} \label{Tesla}
   \|x(\tau)\|^2-\|x(0)\|^2 \leq\m 2\m\Re \int_0^\tau \langle
   u(t),y(t) \rangle \dd t.
\end{equation}
This is equivalent to the fact that, for almost every $t\geq 0$,
\begin{equation} \label{Energix}
   \frac{\dd}{\dd t}\m \|x(t)\|^2 \m\leq\m 2\m\Re\langle u(t),y(t)
   \rangle \m.
\end{equation}

We denote by $\cline_0$ the open right half-plane in $\cline$. For any
bounded operator $T$ acting on a Hilbert space, we denote $\Re T=
\half(T+T^*)$. The inequality \rfb{Tesla} implies that $\PPP_u$ is
analytic on $\cline_0$ and (see \cite{staffans2002passive})
\begin{equation} \label{Hangzhou}
   \Re \PPP_u(s) \m\geq\m 0 \FORALL s\in \cline_0 \m.
\end{equation}
Transfer functions with the property \rfb{Hangzhou} are called {\em
positive}. For a beautiful overview of positive transfer functions
we refer to \cite{guiver2017transfer}.

{
A linear and time-invariant DPS with state space $X$ whose operator
semigroup $\tline$ satisfies $\tline_t x_0\to 0$ for any $x_0\in X$,
is called \emph{strongly stable} (for the standard terminology and
background theory, see \cite[Definition 5.4.1]{RCurtainHZwart-book},
\cite[V.Definition 1.1]{EngNag}, and \cite{art69}). }

Our goal is to design a controller $\Sig_c$ that strongly stabilizes
the system and makes the tracking error converge to zero, while having
only little information about the plant. $\Sig_p$ should be
interconnected with $\Sig_c$, as in Fig.~\ref{fig:blocks}, leading to
the closed-loop system $\Sig_L$ with inputs $w,v,r$ and outputs
$y,u^{out}$. $\Sig_c$ should solve the \textit{output regulator
problem}, stated below.

\textbf{Output regulator problem:} find $\Sig_c$ such that
\begin{enumerate}
   \setlength{\itemsep}{1pt} \setlength{\parskip}{1pt}
   \item the closed-loop system $\Sig_L$ from Fig.~\ref{fig:blocks}
      is well-posed and strongly stable;
   \item any state trajectory of \m $\Sig_L$ is bounded, in the
      presence of $w,v$ and $r$ as in \rfb{exosig};
   \item the tracking error $e$, when low-pass filtered, converges to
      zero.
\end{enumerate}

We now give some further background and comments about this
formulation of the output feedback regulator problem, tailored for
strongly stabilizable systems.

It is well known that many well-posed systems cannot be exponentially
stabilized by a well-posed error feedback controller. Since the 1970s,
researchers have searched for conditions under which a linear DPS is
exponentially stabilizable, see for instance \cite{datko1971linear},
\cite{jacob1999equivalent}, \cite{REBARBER1990241}. There is a theorem
in \cite{art34} which characterizes the necessary and sufficient
condition for a well-posed system to be exponentially stabilized by a
controller $\Sig_c$:

\begin{theorem} \label{Theorem:expon_stabi_estimatable}
Let $\Sig_p$ be a well-posed linear system and $\Sig_c$ a controller
as in Fig.~\ref{fig:blocks}. $\Sig_c$ is a stabilizing controller
for $\Sig_p$ if and only if the following conditions hold:
\begin{enumerate}
   \item Both $\Sig_p$ and $\Sig_c$ are optimizable and estimatable.
   \item The closed{-}loop system is input-output stable.
\end{enumerate}
\end{theorem} \vspace{2mm}

For $\Sig_p$ optimizability is understood with respect to the control
input $u$. Thus, if $\Sig_p$ would be exponentially stabilized by a
well-posed controller, then it would be optimizable and estimatable
(as defined in the cited reference).
There are many well-posed systems that are not optimizable or not
estimatable. For such systems, strong stabilization is the best that
we can hope for. For example, \cite{morris2014modeling} have
considered a voltage-actuated piezoelectric beam with magnetic effects
and proved that for almost all system parameters, the system is not
optimizable but it is strongly stabilizable. Strong stabilization by
colocated (static) feedback has been much studied, see
{\cite{art69,Lassi2025JFA}} and the references therein.

\begin{figure} \centering 
   \includegraphics[width=0.4\textwidth]{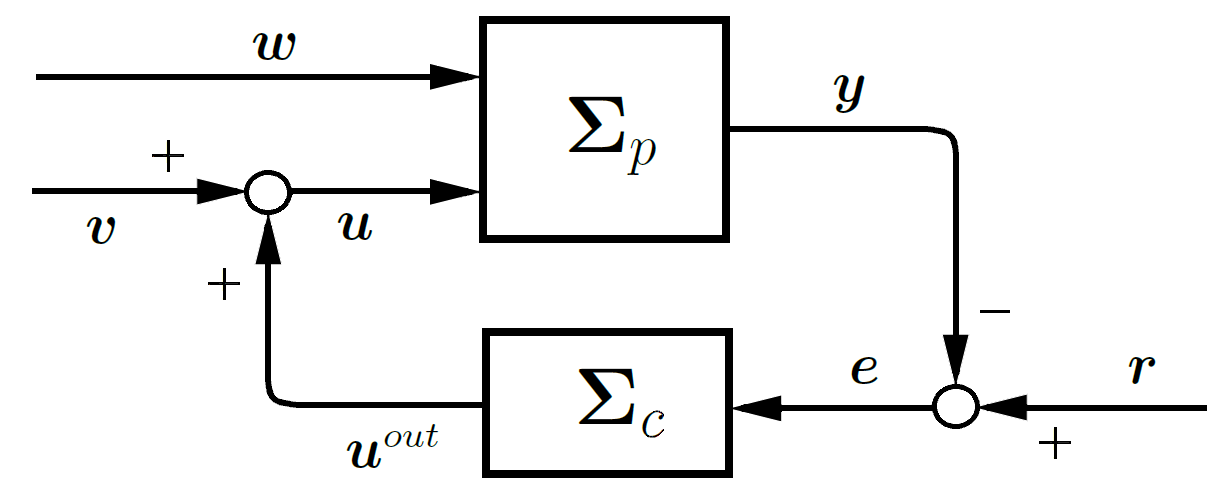}
   \caption{The well-posed plant $\Sig_p$ interconnected with the
   impedance passive \textbf{IMC} $\Sig_c$. The transfer functions of
   these systems are $\PPP=[\PPP_w\ \ \PPP_u]$ and $\gg$. The system
   $\Sig_p$ is assumed to be impedance passive from $u$ to $y$.}
   \label{fig:blocks}
\end{figure}

There are some problems for DPS with defining what exactly we mean by
$e$ converging to zero, because $y$, being the output signal of \m
$\Sig_p$ with a possibly unbounded observation operator, does not have
to be continuous, it is only known to be locally $L^2$. The same is
true for $e$, so that point evaluations of $e$ do not make sense, in
general. This problem has been encountered in several references, for
example, in {\cite{paunonen2016,MR3285896,art59}}, where
they replace the condition $\lim_{t\rarrow +\infty}e(t)=0$ with $e\in
L^2_{\alpha}([0,\infty);\cline^p)$, with $\alpha<0$. This condition
means that the function $t\mapsto{\rm e}^{-\alpha t}e(t)$ is in
$\LE{\cline^p}$. However, this concept is very strong, and it is
suitable for solutions of the regulator problem that lead to an
exponentially stable closed-loop system. We will work with a less
demanding concept of convergence to zero, defined as follows:
\vspace{1mm}

\begin{definition} \label{Nasrallahs_virgins}
Let $v\in L^1_\loc([0,\infty);\cline^p)$. We say that $v$ tends to
zero when filtered if for some $T>0$, the output of the low-pass
filter with transfer function $1/(1+Ts)$ {subject to the input
$v$, tends to zero as $t\to\infty$.}
\end{definition} \vspace{2mm}

By applying a scalar filter to a vector signal we mean, of course,
that we apply the filter componentwise. {This concept is
almost equivalent to the convergence to zero of the moving averages
of $\|v\|$, as used (in a similar context) in
\cite{Paun2017rob,Lassi2019}, see Propositions \ref{LaDefense} and
\ref{Fordo} below.} \vspace{2mm}

\begin{assumption} \label{ass_trans}
Let $A$ be the semigroup generator of $\Sig_p$. We assume that
$i\o_j\in\rho(A)$ and $\PPP_u(i\o_j)$ is invertible, for all
$j\in\Jscr$.
\end{assumption}

It is well known that for the regulator problem to be solvable, the
transmission zeros of the plant should not coincide with any
eigenvalue of the exosystem, see Assumption 2 in \cite{paunonen2016}
or Theorem 5.2 in \cite{art66}. Assumption \ref{ass_trans} implies
this condition. {Our Assumption \ref{ass_trans} is equivalent to
\cite[Assumption 4.1]{Lassi2019}.}

Our main contribution is: We construct a finite-dimensional
controller that solves the output regulation problem for impedance
passive well-posed systems satisfying Assumption \ref{ass_trans},
using only error feedback. Notably, this result holds even for systems
possessing infinitely many eigenvalues on the imaginary axis.

Apart from \cite{art114,art136}, the works most closely related to
this one are: \cite{GuiverLogeman} derived a stability theorem, more
general than the classical passivity and small-gain theorems, for the
feedback connection of two (possibly infinite dimensional)
time-invariant linear systems in the spirit of combined
passivity/small-gain results. They have shown that this stability
theorem is applicable for output regulation for periodic signals by
repetitive control. They build on characterizations of
positive-realness and on stability properties after various types of
feedback, developed in \cite{guiver2017transfer}. Another closely
related reference is \cite{Lassi2019}, which (like us) treats the
regulator problem for impedance passive well-posed linear systems.
\cite{Lassi2019} considers a broader range of topics than our paper:
it covers also the possible non-uniform (in particular, polynomial)
or exponential stability of the closed-loop system, and the set
$\Jscr$ of frequencies in \rfb{exosig} is allowed to be infinite (in
particular, repetitive control is covered). When restricting our
attention to the regulator problem with strong stabilty, as in our
paper, the sets of assumptions in the present work and in
\cite{Lassi2019} 
{represent different trade-offs,} see the
comments after Theorem \ref{theorem:stab}.

The remaining {\em contents of this paper} are as follows: In
Sect.~\ref{sec2} we study signals that tend to zero when filtered. In
Sect.~\ref{sec3} we consider the disturbance rejection and output
tracking problem for an impedance passive system $\Sig_p$ with a
finite{-}dimensional controller. We assume that the semigroup
generator $A$ of the plant $\Sig_p$ has at most a countable set of
spectral points on the imaginary axis, in which the eigenvalues are
observable.  The controller $\Sig_c$ is a finite{-}dimensional
\textbf{IMC}, tuned to the known disturbance frequencies $\o_j$ (that
is, it has poles at the points $i\o_j$). The \textbf{IMC} is designed
to be impedance passive and minimal. These constraints on the
\textbf{IMC} are not restrictive, it is easy to construct an
\textbf{IMC} that satisfies them. We show that the designed $\Sig_c$
solves the regulator problem.  In fact, we get the stronger statement
that the function $t\rarrow\| e(t)\|$ tends to zero when filtered. In
Sect.~\ref{sec4} we illustrate our results with two examples.

\textbf{Notation:} The set of all the bounded linear operators from
$X$ to $Y$ is denoted by $\Lscr(X,Y)$. If $Y=X$, then we write
$\Lscr(X)$ instead of $\Lscr(X,X)$. $\cline_\beta$ is the open
half-plane $\cline_\beta=\{s\in\cline|\m \Re s>\beta \}$. For any
Banach space $Z$, $H^\infty_\beta(Z)$ denotes the Hardy space
$$ H^\infty_\beta(Z) \m=\m \Big\{G:\cline_{\beta}\rarrow Z\ |\ G
   \text{ analytic and bounded} \Big\}.$$
When the space $Z$ is clear from the context, we write
$H^\infty_\beta$ for $H^\infty_\beta(Z)$. $\sigma(A),\sigma_p(A)$ and
$\rho(A)$ are, respectively, the spectrum, point spectrum and
resolvent sets of $A$.

\section{Signals converging to zero when filtered}
\setcounter{equation}{0} \label{sec2} 

In this section, we take a closer look at the concept of a signal that
tends to zero when filtered. The main result of this section is that
if a well-posed system is strongly stable, and its input signal is
zero, then its output signal tends to zero when filtered.

Signals that tend to zero when filtered have been introduced in
Definition \ref{Nasrallahs_virgins}. For consistency, we have to prove
that indeed the concept introduced here does not depend on the time
constant $T$ of the filter.

\begin{proposition} \label{missile}
Let $v\in L^1_\loc([0,\infty);\cline^p)$ and $T>0$ be such that
the continuous signal
\begin{equation} \label{ballistic}
  { v_T(t)} \m=\m \frac{1}{T} \int_0^t e^{-(t-\sigma)/T}
  v(\sigma) \dd \sigma
\end{equation}
converges to zero (as $t\rarrow\infty$). If we replace $T$ with any
$\tau>0$, the resulting signal {$v_\tau$} again converges to zero.
\end{proposition}

\begin{pf} \m The transfer function leading from {$v_T$ to $v_\tau$}
is \m $\HHH(s)=\frac{1+Ts}{1+\tau s}$, meaning that ${{\hat{v}}_\tau}
(s)=\HHH(s) {\hat{v}_T}(s)$. We decompose \vspace{-2mm}
$$ \HHH(s) \m=\m \frac{T}{\tau} + \frac{1-\frac{T}{\tau}}
   {1+\tau s} \m,$$
so that {$v_\tau$} can be decomposed into $\frac{T}{\tau}{v_T}$ plus
$1-\frac{T}{\tau}$ times a low-pass filtered version of {$v$}.
Clearly, both components tend to zero, hence $\lim_{t\rarrow\infty}
{v_\tau}(t)=0$. \hfill $\square$ \end{pf}

\vspace{-2mm}
From an engineering point of view, we can give the following
justification: any measuring or recording system encountered in
engineering has a finite bandwidth. Thus, when we look at a signal,
for instance on an oscilloscope, we are always looking at a low-pass
filtered version of the signal. If the bandwidth $1/T$ of the filter
is large, then we tend to ignore the presence of this low-pass filter.
When we see that a signal tends to zero, what we actually see is that
a low-pass filtered version of the signal tends to zero. For signals
of class $L^2_{\rm loc}$, point evaluations do not make sense, and
hence we cannot define convergence to zero in the classical way.

\begin{proposition} \label{Hodeida}
If $v\in\LE{\cline^p}$, then $v$ tends to zero when filtered.
\end{proposition}

\vspace{-2mm}
\begin{pf} \m Thanks to Proposition \ref{missile}, we may take
$T=1$ in \rfb{ballistic}. We decompose the function ${v_T}$ from
\rfb{ballistic}:
$$ \m\ \ \ {v_T}(2t) \m=\m \int_0^t {\rm e}^{-(2t-\sigma)}
   v(\sigma) \dd\sigma + \int_t^{2t} {\rm e}^{-(2t-\sigma)} v(\sigma)
   \dd \sigma \m,$$
whence, using the Cauchy-Schwarz inequality,
\[ \begin{aligned}
   &\|{v_T}(2t)\| \leq \frac{{\rm e}^{-2t}}{\sqrt{2}} \left(
   {\rm e}^{2t}-1\right)^\half \left(\int_0^t \|v(\sigma)\|^2
   \dd \sigma \right)^\half \\ &\quad + \frac{{\rm e}^{-2t}}
   {\sqrt{2}} \left( {\rm e}^{4t}-{\rm e}^{2t} \right)^\half
   \left(\int_t^{2t} \|v(\sigma)\|^2 \dd \sigma \right)^\half
\end{aligned} \]
\[ \begin{aligned}
   & \leq \frac{{\rm e}^{-t}}{\sqrt{2}} \left(\int_0^t
   \|v(\sigma)\|^2 \dd \sigma \right)^\half + \frac{1}{\sqrt{2}}
   \left( \int_t^{2t} \|v(\sigma)\|^2 \dd\sigma \right)^\half \\
   & \leq \frac{{\rm e}^{-t}}{\sqrt{2}} \|v\| + \frac{1}
   {\sqrt{2}} \left( \int_t^\infty \|v(\sigma)\|^2 \dd \sigma
   \right)^\half.
\end{aligned} \]
The last integral above tends to zero as { $t\rarrow \infty$}, and
hence $\lim_{t\rarrow\infty}\|{v_T}(t)\|=0$. \hfill $\square$
\end{pf}

\vspace{-2mm}
There are examples of functions $v:[0,\infty)\rarrow\rline$ that are
continuous and do not tend to zero (when $t\rarrow\infty$), are not in
$L^2[0,\infty)$ but they tend to zero when filtered, for instance
$v(t)=\sin (t^2)$. There are also functions $v:[0,\infty)\rarrow
\rline$ that are continuous and positive, do not tend to zero (when
$t\rarrow\infty$), do not belong to $L^2[0,\infty)$, but they tend to
zero when filtered, for instance $v(t)=|\sin t|^t$. For the proofs of
these facts, see the Appendix. The following proposition should give a
bit more insight into functions that converge to zero when filtered.

\begin{proposition} \label{LaDefense}
Suppose that $v\in L^1_\loc([0,\infty);\cline^p)$ tends to zero
when filtered, $\tau>0$, and define\\ ${g}:[0,\infty)\rarrow
\cline^p$ by \vspace{-2mm}
\begin{equation} \label{Kiel}
  {g}(t) \m=\m \frac{1}{\tau} \int_{t-\tau}^t v(\sigma) \dd\sigma
\end{equation}
(moving averages of $v$). Then $\lim_{t\rarrow\infty} {g}(t)=0$.
\end{proposition}

\vspace{-2mm}
\begin{pf} \m Define ${v_\tau}$ by \rfb{ballistic}, with
$T=\tau$. According to Proposition \ref{missile},
$\lim_{t\rarrow\infty}{v_\tau}(t)=0$. In terms of Laplace
transforms we have
$$ {\hat{g}}(s) \m=\m \frac{1-e^{-\tau s}}{\tau s}\hat
   v(s) \m=\m \frac{1-e^{-\tau s}}{\tau s}(1+\tau s)
   {\hat v_\tau}(s) \m.$$
We see from here that we can decompose ${g}={g}_1+
{g}_2$, where ${g}_1(t)$ is the (moving) average of
${v_\tau}$ over $[t-\tau,t]$, while ${g}_2(t)=
{v_\tau}(t)-{v_\tau}(t-\tau)$. Clearly both components of
${g}$ tend to zero as $t\rarrow\infty$. \hfill $\square$ \m
\end{pf}

\vspace{-2mm}
The following is a partial converse of Proposition \ref{LaDefense}.

\begin{proposition} \label{Fordo}
Suppose that $v\in L^1_\loc[0,\infty)$, with {non-negative
values.} Let $\tau>0$ and define its moving averages \
${g}\nm:\nm[0,\infty)\rarrow[0,\infty)$ by \rfb{Kiel}. If \
$\lim_{t\rarrow\infty} {g}(t)=0$, then $v$ tends to zero when
filtered.
\end{proposition}

\vspace{-2mm}
\begin{pf} \m Define ${v_\tau}$ by \rfb{ballistic}, with
$T=\tau$. We have to show that ${v_\tau}$ converges to zero. We
have for any $n\in\nline$
$$ {v_\tau}(n\tau) \m=\m \frac{1}{\tau}\sum_{k=1}^n \int_{(k-1)
   \tau}^{k\tau} e^{-(n\tau-\sigma)/\tau} v(\sigma) \dd\sigma$$
$$ \leq\m \frac{1}{\tau} \sum_{k=1}^n e^{-(n-k)} \int_{(k-1)\tau}^{k
   \tau} v(\sigma)\dd\sigma \m=\m \sum_{k=1}^n e^{-(n-k)}g(k\tau)\m.$$
Thus, the sequence $(v_\tau(n\tau))$ is the convolution product of the
sequences $(e^{-n})$ and $(g(n\tau))$, the latter being convergent to
zero. It follows that $\lim_{n\rarrow\infty}v_\tau(n\tau)=0$, and from
here the statement follows with ease. \hfill $\square$ \m \end{pf}

For the remainder of this section, $X$ and $Y$ are Hilbert spaces,
$\tline$ is a strongly continuous semigroup on $X$, with generator
$A$, and $C:\Dscr(A)\rarrow Y$ is an {\em admissible observation
operator} for $\tline$. This means that for some (hence, for any)
$\tau>0$, the operator $\Psi_\tau:\Dscr(A)\rarrow L^2([0,\tau],Y)$,
\begin{equation} \label{Jeremy_Bowen}
  (\Psi_\tau x_0)(t) \m=\m C\tline_t x_0 \FORALL t\in [0,\tau]
  \vspace{1mm}
\end{equation}
has an extension in $\Lscr(X,L^2([0,\tau],Y))$. We denote also this
extension by $\Psi_\tau$. The function $y\in\LEloc{Y}$ which on any
interval $[0,\tau]$ coincides with $\Psi_\tau x_0$ is called the {\em
output function corresponding to the initial state} $x_0$ (see
\cite{art10} for details), and we write $y=\Psi_\infty x_0$.
$\Psi_\infty$ is called the {\em extended output map} of $(A,C)$
(\cite[Sect.~2]{art10}).

\begin{theorem} \label{Macron}
With $X, Y, A,\tline$ and $C$ as above, assume that $\tline$ is
strongly stable. For some initial state $x_0\in X$ let $y$ be the
corresponding output function. {Then the function $t\mapsto
\|y(t)\|^2$ (and hence also the function $t\mapsto\|y(t)\|$)
tends to zero when filtered.}
\end{theorem}

The conclusion of this theorem is stronger than the statement that
$y(t)$ tends to zero when filtered. For example, the function
$y(t)=\sin(t^2)$ (mentioned earlier) tends to zero when filtered, but
$|y(t)|$ does not tend to zero when filtered. Indeed, if it would,
then according to Proposition \ref{LaDefense} the moving averages of
$|y(t)|$ would converge to zero, and it is easy to see that this is
not true.

\vspace{-5mm} {
\begin{pf} \m For $x_0$ and $y$ as in the statement and $t\geq 0$,
define \vspace{-2mm}
$$ g(t) \m=\m \int_t^{t+1} \|y(t)\|^2 \dd t \m=\m \|\Psi_1
   \tline_t x_0\|^2 \m.$$
Clearly $g(t)\leq\|\Psi_1\|^2\cdot\|\tline_t x_0\|^2$, so that
the strong stability of $\tline$ implies that $g(t)\rarrow 0$.
According to Proposition \ref{Fordo} (with $\tau=1$ and $v(t)=\|y(t)
\|^2$), we conclude that $t\mapsto\|y(t)\|^2$ tends to zero when
filtered. Since (by Cauchy-Schwarz) $\int_t^{t+1}\|y(t)\|\dd t\leq
\left(\int_t^{t+1}\|y(t)\|^2\dd t\right)^\half$, we get the same
conclusion for $t\mapsto\|y(t)\|$. \hfill$\square$ \end{pf}}

\section{Error feedback controller design}
\setcounter{equation}{0} \label{sec3} 

In this section we solve the output regulator problem (with strong
stability, as defined in Sect.~\ref{sec1}) for an impedance passive
well-posed linear system $\Sig_p$, with the feedback structure as
shown in Fig.~\ref{fig:blocks}. First, we consider the strong
stability of $\Sig_p$ with an \textbf{IMC} connected in feedback. This
strong stability means that, when all the disturbance and reference
signals are zero, the states of both subsystems converge to zero. The
\textbf{IMC} $\Sig_c$ is finite{-}dimensional, minimal, and
impedance passive.

\begin{figure} \centering 
   \includegraphics[width=0.4\textwidth]{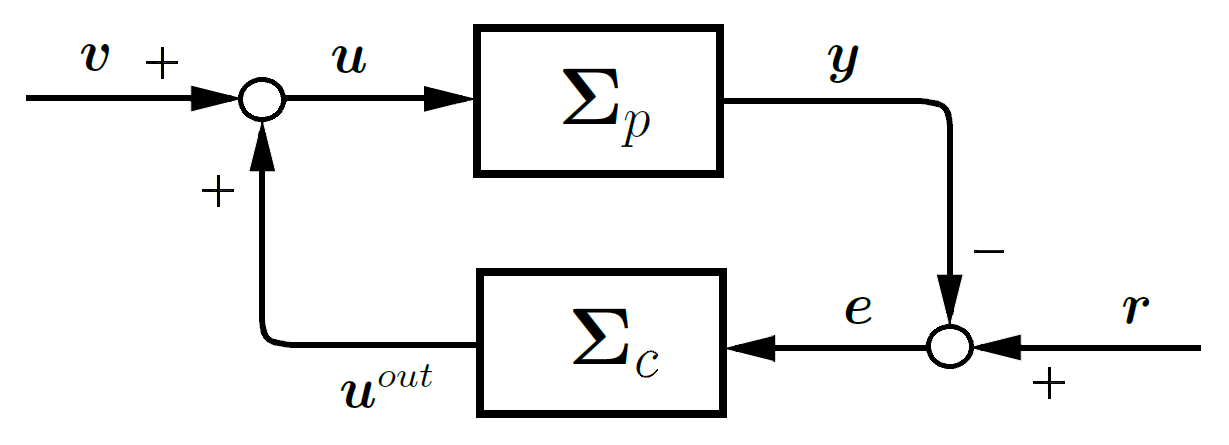}
   \caption{A coupled system $\Sig_K$ consisting of a well-posed and
   impedance passive system $\Sig_p$ and a finite{-}dimensional
   internal model controller $\Sig_c$, connected in feedback.}
   \label{fig:2}
\end{figure}

\begin{lemma} \label{lemma_sensity}
Let $H$ be a Hilbert space and $T\in\Lscr(H)$. Then \vspace{-2mm}
$$ \Re\m T \m\geq\m \delta I, \ \ \ \delta>0 $$
if and only if there exists $Q\in\Lscr(H)$ such that
$$ T \m=\m 2\delta (I-Q)^{-1}, \ with \ \ \ ||Q || \m\leq\m 1 \m.$$
Furthermore, if the above conditions are satisfied and the number
$M>0$ is such that $M\geq\|T\|/(2\delta)$, then
\begin{equation} \label{Hodeida-LemAdd}
  \Re\m Q \m\leq\m \left(1-\frac{1}{2M^2} \right) I \m.
\end{equation}
\end{lemma}

This is an easy consequence of Lemma 3.3 in \cite{art59} and
its proof is omitted.

\begin{definition}
Let $\tline$ be a strongly continuous semigroup on $X$, with generator
$A$, and let $C$ be an admissible observation operator for $\tline$. A
point $p\in\cline$ is called an {\em observable eigenvalue} of the
pair $(A,C)$ if
\begin{enumerate}
\item $p$ is an eigenvalue of $A$, with eigenspace $E_p\subset X$.
\item For any nonzero $x_0\in E_p$, the function $t\rarrow C\tline_t
  x_0$ is not identically zero.
\end{enumerate}
The second condition is equivalent to $Cx_0\neq 0$, because $C\tline_t
x_0=e^{pt}C x_0$ for all $t\geq 0$.
\end{definition}

In the following two theorems, for the sake of simplicity, we ignore
the input $w$ from Fig.~\ref{fig:blocks} (we set it to zero), so that
we obtain the feedback system shown in Fig.~\ref{fig:2}. These
theorems are a generalization of \cite[Theorem 1.2]{art59}, where it
was assumed that $\Sig_p$ is exponentially stable. A more general
version (but less general than the theorems below) appears in the
conference paper \cite{art136}, but there the key part of the proof
is not given.

\begin{theorem} \label{theorem:well}
Let $\Sig_p$ be a well-posed impedance passive system with input
space and output space $U=\cline^p$, state space $X$, semigroup
generator $A$, control operator $B$, observation operator $C$ and
transfer function $\PPP_u$. Let $\Sig_c$ be an impedance passive and
minimal realization of
\begin{equation} \label{tf_controller}
   \gg(s) \m=\m \sum_{j=1}^n \dfrac{c_j^2}{s - i\o_j} + d \m,
\end{equation}
with state space $X_c$, system operator $a\in\Lscr(X_c)$, where $c_j
\in\Lscr(U),\m c_j>0,\m d\in\Lscr(U)$ and $\Re\m d>0$.

Under Assumption \ref{ass_trans}, the feedback system \m $\Sig_K$ in
Fig.~\ref{fig:2}, with input $\sbm{v\\ r}$ and output $\sbm{y \\
u^{out}}$, with state space $X\times X_c$, is well-posed, input-output
stable (i.e., its transfer function $\GGG^K$ is in $H^\infty_0$) and
impedance passive.

The sensitivity $S=(I+\PPP_u\gg)^{-1}$ is in $H^\infty_0(\Lscr(U))$.
\end{theorem}

\vspace{-2mm}
\begin{pf} {\bf Step 1.} {One possible minimal realization of
$\gg$ from \rfb{tf_controller} is given by the feedthrough operator
$d$ satisfying $\Re d>0$, along with the matrices:}
\begin{align} \label{minreal}
\begin{split} & a \m=\m \mathrm{diag}\{ i\o_1 I,\; i\o_2I,\ \ldots
   i\o_nI\}, \\ & b=[c_1\ c_2\ \ldots c_n]^\top, \ \ \ c \m=\m
   [c_1\ c_2\ \ldots c_n]. \end{split}
\end{align}
This realization is impedance passive (easy to verify), and {
every minimal realization of $\gg$ is isomorphic to it via a boundedly
invertible transformation. It is a fundamental property that such an
isomorphism preserves the well-posedness, input-output stability,
generator spectral properties, and strong stability of the closed-loop
system. Thus, it suffices to proceed with this realization.}

The closed-loop system $\Sig_K$ in Fig.~\ref{fig:2} can be considered
as being obtained from the ``open-loop" system $\Sig_o$ consisting of
$\Sig_p$ and $\Sig_c$ acting separately with input $\sbm{u \\ e}$ and
output $\sbm{y \\ u^{out}}$, with state space $X\times X_c$, by
applying the static output feedback operator $K=\sbm{ 0 & I \\ -I &
0}$\m, in the spirit of \cite{art12}, \cite{art34}. The generating
operators of \m $\Sig_o$ are:
\begin{align} \label{Oshrit}
   A^o \m=\m \bbm{A & 0 \\ 0 & a}, \ \ B^o \m=\m \bbm{ B & 0 \\
   0 & b}, \ \ C^o \m=\m \bbm{C & 0 \\ 0 & c},
\end{align}
where $A^o$ is the semigroup generator, $B^o$ is the control operator
and $C^o$ is the observation operator. It is easy to see that $\Sig_o$
is impedance passive and its transfer function from the input
$\sbm{u\\ e}$ to the output $\sbm{y\\ u^{out}}$ is \vspace{-1mm}
$$ \GGG^o \m=\m \bbm{\PPP_u & 0 \\ 0 & \gg}.$$
The transfer function of $\Sig_K$ from the input $\sbm{v\\ r}$ to the
output $\sbm{y\\ u^{out}}$ is \m $\GGG^K=\GGG^o(I-K\GGG^o)^{-1}$, \m
whence
\begin{align} \label{tf:input_to_output}
   \m\n \GGG^K = \bbm{(I+\PPP_u \gg)^{-1}\PPP_u &
   (I+\PPP_u \gg)^{-1}\PPP_u \gg\\ -\gg(I+\PPP_u \gg)^{-1}
   \PPP_u  & \gg(I+\PPP_u \gg)^{-1}} .
\end{align}
We denote by $\Psi_\tau$ the output maps of $\Sig_p$, defined (for any
$\tau\geq 0$) as in \rfb{Jeremy_Bowen}, and we denote by $\Psi^o_\tau$
and by $\Psi^K_\tau$ the corresponding operators for $\Sig_o$ and for
$\Sig_K$. It is well known (see \cite[formula (6.13)]{art12}) that we
have
\begin{equation} \label{JingJingShebek}
   \Psi^o_\tau \m=\m (I-\fline^o_\tau K)\Psi^K_\tau \m,
\end{equation}
where $\fline^o_\tau$ is the input-output operator of $\Sig_o$ on the
interval $[0,\tau]$ (see, for instance, \cite{art10}).

{\bf Step 2.} \m Here we show that $S$ is well defined and analytic on
$\cline_0$ and we derive an estimate for $\|S(s)\|$. It is easy to see
from $c_j>0$ that for every index $j\in\Jscr$,
\begin{equation} \label{ineq:positive}
   \m\ \ \ \ \ \Re\m\left(\frac{c_j^2}{s-i\o_j}\right) \m>\m 0
   \FORALL s\in\cline_0 \m.
\end{equation}
Therefore, from \rfb{tf_controller} we see that
\begin{equation} \label{ineq:g_positive}
   \Re\m \gg(s) \m>\m \Re\m d \FORALL s\in\cline_0 \m.
\end{equation}
Let $\delta\in(0,\infty)$ be a lower bound for the positive matrix
$\Re d$, i.e., $\Re d\geq\delta I$. Then from \rfb{ineq:g_positive}
we have
\begin{equation} \label{Boulder}
   \Re\m \gg(s) \m>\m \delta I \FORALL s\in\cline_0 \m.
\end{equation}
This implies also that for all $s\in\cline_0$, \m $\|\gg(s) v\|>
\delta\|v\|$, for all $v\in U$, and hence
\begin{equation} \label{ineq:g_norm}
   \|\gg(s)\|\geq \delta, \quad \|\gg^{-1}(s)\| \m\leq\m
   \frac{1}{\delta} \FORALL s\in\cline_0 \m.
\end{equation}

By Lemma \ref{lemma_sensity}, the inequality \rfb{Boulder} implies
that for each $s\in\cline_0$ there exists $Q(s)\in\Lscr(U)$ such that
$$ \gg(s) \m=\m 2\delta({I}-Q(s))^{-1}, \ \ \ ||Q(s)||\leq 1.$$
Let $\widetilde S=I+\PPP_u \gg$, which is defined and analytic on
all of $\cline_0$. Since the values of this function are square
matrices, to show that $\widetilde S(s)$ is invertible, it is
enough to show that it is bounded from below. We have
\begin{align} \label{sensitivity}
  \widetilde S(s) \m&=\m I+2\delta\PPP_u(I-Q(s))^{-1} \notag \\ &=\m
  [I-Q(s) + 2\delta\PPP_u(s)] (I-Q(s))^{-1} \m \notag \\ &=\m
  (2\delta)^{-1}[I-Q(s)+2\delta\PPP_u(s)] \gg(s) \m.
\end{align}
According to \rfb{Hodeida-LemAdd} we have
$$ \Re Q(s) \m\leq\m \left( 1-\frac{2\delta^2}{\|\gg(s)\|^2} \right)
   I \FORALL s\in\cline_0 \m.$$
If we substitute this estimate and \rfb{Hangzhou} into
\rfb{sensitivity}, we get that for all $s\in\cline_0$,
\begin{align*}
   \Re (\widetilde{S}(s)\gg^{-1}(s)) &= (2\delta)^{-1} \Re[I - Q(s) +
   2\delta \PPP_u(s)] \\ & \geq (2\delta)^{-1} \frac{2\delta^2}
   {\|\gg(s)\|^2}I = \epsilon(s) I,
\end{align*}
where $\e(s)=\frac{\delta}{\|\gg(s)\|^2}>0$. It follows that
$$ \|\widetilde S(s) \gg^{-1}(s) v\| \m\geq\m \e(s) \|v\| \ \ \
   \forall\ v\in U=\cline^p,\ \ s\in\cline_0 \m,$$
   or equivalently (using \rfb{ineq:g_norm}),
$$ \|\widetilde S(s) z\| \m\geq\m \e(s) \|\gg(s) z\| \m\geq\m \e(s)
   \delta \|z\| \ \ \ \forall\ z\in U,\ \ s\in\cline_0 \m.$$
This lower bound shows that $\widetilde S(s)$ is invertible for all
$s\in\cline_0$. Since $\tilde{S}$ is analytic, we get that
$S=(\tilde S)^{-1}$ is well defined and analytic on $\cline_0$ and
\begin{equation} \label{Mao_statue}
  \| S(s) \| \m\leq\m \frac{1}{\delta\e(s)} \m=\m
  \frac{\|\gg(s)\|^2}{\delta^2} \FORALL s\in\cline_0 \m.
\end{equation}

{\bf Step 3.} Next we show that $\GGG^K\in H^\infty_0$, i.e., all the
transfer functions in \rfb{tf:input_to_output} are analytic and
bounded on $\cline_0$. These are $\gg S$, $S\PPP_u$, $S\PPP_u\gg$ and
$-\gg S\PPP_u$. First we show that \m $S\in H^\infty_0$. From
\rfb{tf_controller} we can derive that
$$ \lim_{\cline_0\ni s\rarrow i\o_j } (s-i\o_j)[I+\PPP_u(s)\gg(s)]
   \m=\m \PPP_u(i\o_j)c_j^2 \m.$$
According to Assumption \ref{ass_trans} and from $c_j>0$, we know that
the operator on the right-hand side is invertible. The operator on
the left is invertible as well (the inverse is a scalar multiple of
$S(s)$). Therefore, we can write the earlier limit for the inverses:
\vspace{1mm}
\begin{equation} \label{lim:G}
   \lim_{\cline_0\ni s\rarrow i\o_j } \frac{1}{s-i\o_j}S(s) \m=\m
   [\PPP_u(i\o_j)c_j^2]^{-1} \m.
\end{equation}
In particular, it follows that for all $j\in\Jscr$,
\begin{equation} \label{Maddalen_to_Bordeaux}
   \lim_{\cline_0\ni s\rarrow i\o_j } S(s) \m=\m 0 \m.
\end{equation}
This fact implies that there exists $\sigma>0$ such that, denoting
\vspace{-2mm}
$$ \Nscr \m=\m \bigcup_{j\in\Jscr}\Big\{ s\in\cline_0\ |\ |s-i\o_j| <
   \sigma \Big\} \m,$$
we have \m $\|S(s)\|\leq 1$ for all $s\in\Nscr$. From
\rfb{tf_controller} we see that $\gg$ is uniformly bounded on
$\cline_0\setminus\Nscr$, so according to \rfb{Mao_statue} $S$ is
uniformly bounded on $\cline_0\setminus\Nscr$. With the previous
estimate on $\Nscr$, this implies that $S\in H^\infty_0$.

Now we turn to $\gg S$, the transfer function from $r$ to $u^{out}$ in
Fig.~\ref{fig:2}. Using \rfb{tf_controller} and \rfb{lim:G}, we see
that $\gg S$ has a finite limit at $i\o_j$, for each $j\in\Jscr$, so
that $\gg S$ is uniformly bounded on $\Nscr$.  Since $\gg$ is
uniformly bounded on $\cline_0\setminus\Nscr$ and $S\in H^\infty_0$,
we conclude that $\gg S\in H^\infty_0$.

Now we consider $S\PPP_u$, the transfer function from $v$ to $y$. We
have already seen from \rfb{ineq:g_norm} that $\gg^{-1}$ is uniformly
bounded on $\cline_0$. Since \m $S\PPP_u=S\PPP_u\gg\gg^{-1}$ and
$$ S\PPP_u\gg \m=\m (I+\PPP_u\gg)^{-1}\PPP_u \gg \m=\m I-S
   \m\in\m H^\infty_0 \m,$$
we obtain that \m $S\PPP_u\in H^\infty_0$.

Next we handle the transfer function $S\PPP_u\gg$, the easiest case:
Its uniform boundedness follows immediately from the identity
$(I+\PPP_u\gg)^{-1}\PPP_u\gg=I-S$.

Finally, we consider $-\gg S\PPP_u$, the transfer function from $v$ to
$u^{out}$ in Fig.~\ref{fig:2}. We have proved earlier that
\\ $(I+\gg\PPP_u)^{-1}\gg=\gg(I+\PPP_u\gg)^{-1}\in H^\infty_0$.
Multiplying from the right with $\gg^{-1}$ (which is in $H^\infty_0$,
see \rfb{ineq:g_norm}) we obtain that \m $(I+\gg\PPP_u)^{-1}\in
H^\infty_0$. Using the algebraic identity
$$(I+\gg\PPP_u)^{-1} \m=\m I-\gg(I+\PPP_u\gg)^{-1}\PPP_u \m,$$
we obtain that indeed \m $-\gg S\PPP_u\in H^\infty_0$. Thus, we have
shown that indeed $\GGG^K\in H_0^\infty$.

{\bf Step 4.} We prove the well-posedness and passivity of $\Sig_K$.
Since this feedback system is obtained from the well-posed system
$\Sig_o$ via the static feedback operator $K$ and the closed-loop
transfer function $\GGG^K$ is uniformly bounded on $\cline_0$,
according to the feedback theory in {\cite[Section 3]{art12}},
$\Sig_K$ is well-posed. Moreover, the fact that $\GGG^K\in H^\infty_0$
means (by definition) that $\Sig_K$ is input-output stable.

Since $\gg$ is positive, {its impedance passive minimal realization
satisfies (see \rfb{Energix})}
$$ \frac{\dd}{\dd t} \m \| q(t)\|^2 \m\leq\m 2\Re\m \langle
   e(t),u^{{out}}(t) \rangle \m,$$
where $q$ denotes the state of $\Sig_c$. Similarly, since $\Sig_p$
is impedance passive, it satisfies
$$ \frac{\dd}{\dd t} \m \| x(t)\|^2 \m\leq\m 2\Re\m \langle u(t),
   y(t) \rangle \m,$$
where $x$ is the state of $\Sig_p$. Actually, this inequality holds in
the integral sense, as in \rfb{Tesla}. Denoting $u=v+u^{out}$ and
$e=r-y$ and using the above two inequalities, by a straightforward
computation, we obtain that
$$ \frac{\dd}{\dd t} \m (\|x(t)\|^2 + \| q(t)\|^2) \leq
   2\Re\m \left\langle \sbm{ v(t)\\ r(t)},\sbm{ y(t)\\ u^{out}(t)}
   \right\rangle .$$
This estimate holds in the integral sense, as in \rfb{Tesla}. Thus,
the feedback system $\Sig_K$ is impedance passive. \hfill $\square$
\end{pf}

\begin{theorem} \label{theorem:stab}
Under the conditions of Theorem \ref{theorem:well}, assume that \m
{\rm Ker}\m $B=\{0\}$ and $\sigma_p(A)\bigcap i\rline$ consists only
of observable eigenvalues. Then the closed-loop semigroup generator
$A^K$ has no imaginary eigenvalues and
\begin{equation} \label{Maddalen}
   \left[ \sigma(A^K)\cap i\rline\right] \m\subset\m \left[ \sigma(A)
   \cap i\rline \right]
\end{equation}
in particular, $i\o_j\in\rho(A^K)$ for all $j\in\Jscr$. Moreover, $S$
is analytic in a neighborhood of each of the points $i\o_j$ and
$S(i\o_j)=0$ for all $j\in\Jscr$.

If, in addition, $\sigma(A)\bigcap i\rline$ is countable, then the
closed-loop semigroup $\tline^K$ is strongly stable.
\end{theorem}

{The assumption $\Ker B=\{0\}$ is not restrictive, as one can
always replace $U$ with $(\Ker\m B)^\perp=\Ran\m B^*$, without
affecting the system dynamics.}

Related strong stability and output regulation results are in the
important reference \cite{Lassi2019}. In \cite[Theorem 3.2]
{Lassi2019}, $\Sigma_p$ is assumed 
{to be a regular linear system
(a strict subclass of well-posed systems) and strongly stable, and its
transfer function is assumed to satisfy $\Re\PPP_u (i\o_k)>0$ while
allowing a controller without feedthrough ($d=0$). Furthermore, when
$\Re d>0$, \cite[Remark 3.1]{Lassi2019} aligns with our invertibility
condition but accommodates the case where $i\o_j\in \sigma(A)$, which
is excluded in our framework. On the other hand, condition (3) in
\cite[Theorem 3.2]{Lassi2019} imposes a complex condition on the
resolvent sets of an infinite family of closed-loop semigroups. In
contrast, our framework relaxes the assumptions on the (open-loop)
plant by not requiring it be regular or stable. Moreover, we utilize a
classical observability property instead, which is generally easier to
verify in specific DPS.} In the first part of \cite[Theorem
5.2]{Lassi2019} the assumptions are not explicitly listed, but a
careful reading shows that $\Sig_p$ is regular and impedance passive,
and there exists a feedback operator $D_{c2}>0$ such that the
corresponding closed-loop semigroup is strongly stable and has no
spectrum on the imaginary axis. The concept of error converging to
zero used in \cite{Lassi2019} is different, but thanks to Propositions
\ref{LaDefense} and \ref{Fordo}, the convergence statements are
equivalent. \vspace{-2mm}

\begin{pf} {\bf Step 1.} We show that $A^K$ has no eigenvalue on
$i\rline$. For this, assume that $i\o\in\sigma_p(A^K)$ for some $\o\in
\rline$, and we show that this leads to a contradiction. We use the
notation from \rfb{Oshrit}. First we need to clarify a technical
detail: From Ker $B=\{0\}$ and the structure of $b$ in \rfb{minreal}
we see that Ker $B^o=\{0\}$. This is needed to apply (in the next
paragraph) a result from \cite{art45}.

{\bf First case of $i\o\in\sigma_p(A^K)$.} We consider the case
$\o\neq\o_j$ for all $j\in\Jscr$ and $i\o\in\rho(A)$, so that
$i\o\in\rho(A^o)$. According to Theorem 1.1 in \cite{art45}, if $i\o
\in\sigma_p(A^K)$, then Ker $[I-K\GGG^0(i\o)]\neq\{0\}$, so that
$$ \det[I-K\GGG^0(i\o)] \m=\m 0,$$
hence $[I-K\GGG^0]^{-1}$ has a pole at $i\o$. This implies that
$[I-K\GGG^0]^{-1}$ is not in $H^\infty_0$. Since $[I-K\GGG^0]^{-1}$ is
the transfer function from $\sbm{v \\ r} $ to $\sbm{u \\ e}$ and, as
already proved, the feedback connection of $\PPP_u$ and $\gg$ is
input-output stable, this is a contradiction.

{\bf Second case of $i\o\in\sigma_p(A^K)$.} Now we consider the case
$i\o\in\sigma(A)$ (hence, by Assumption \ref{ass_trans}, $\o\not=\o_j$
for all $j\in\Jscr$). From the standard feedback theory of well-posed
systems (Proposition 6.6 of \cite{art12})
$$ A^K \m=\m A^o+B^o K C^K \ \ \ \text{on}\ \ \Dscr(A^K),$$
where $C^K$ the observation operator of $\Sig_K$. Thus, the
eigenvector of $A^K$ is a non-zero $\sbm{x_p\\ q_p}\in\Dscr(A^K)$ such
that
\begin{align} \label{eq_characteristic}
   [A^o+B^oK C^K] \sbm{x_p\\ q_p} \m=\m i\o
   \sbm{x_p\\ q_p}.
\end{align}
Denote \vspace{-2mm}
\begin{align} \label{denote:x_p_q_p}
  \sbm{y_p\\ u_p^{out}} \m=\m C^K\sbm{x_p\\ q_p},
\end{align}
then substitute \rfb{denote:x_p_q_p} into \rfb{eq_characteristic},
leading to
\begin{equation} \label{eq:eigenAB}
   \begin{array}{l} A x_p+B u^{out}_p \m=\m i\o\m x_p,\ \qquad
   a q_p - b y_p \m=\m i\o\m q_p. \end{array}
\end{equation}

Let $\sbm{\xx\\ \qq}$ be the state trajectory of $\Sig_K$
corresponding to the initial state $\sbm{x_p\\ q_p}$ and the input
$\sbm{v\\ r}=0$ and let $\sbm{\yy\\ \uu^{out}}$ be the corresponding
output function. It follows from \rfb{eq_characteristic} that
$\xx(t)=e^{i\o t}x_p$ and $\qq(t)= e^{i\o t}q_p$. It is clear from
Fig.~\ref{fig:2} that for all $t\in[0,\tau]$, \ $\yy(t)=\CD
\sbm{\xx(t)\\ \qq(t)}$ and $\uu^{out}(t) =c\qq(t)-d\yy(t)$, where
$\CD$ is the combined observation/feedthrough operator of \m
$\Sig_p$, see for instance \cite[Sect.~3]{art11}. Taking $t=0$, we
obtain
\begin{equation} \label{eq:upUyp}
   y_p \m=\m [C \& D]\sbm{x_p \\ u^{out}_p}, \ \ \
   u^{out}_p \m=\m c q_p-d y_p.
\end{equation}
Thus, from \rfb{eq:eigenAB} and \rfb{eq:upUyp},
\begin{align}
   \begin{bmatrix} \begin{matrix} A & B \end{matrix} \\[-2pt]
   \begin{matrix} -[C \& D] \end{matrix} \end{bmatrix}
   \left[\begin{array}{l} x_p \\ u^{out}_p \end{array} \right]
   \m=\m \left[ \begin{array}{c} i\o x_p \\ -y_p
   \end{array} \right] \m.
\end{align}
According to the well-known characterization of impedance passive
systems, see {\cite[Theorem 4.2]{staffans2002passive},}
$\sbm{A & B\\ -[C & \& D]}$ is m-dissipative. Hence
$$ \Re \m \left\langle\left[\begin{array}{l} x_p \\ u^{out}_p
   \end{array}\right],\left[\begin{array}{c} i\o x_p \\ -y_p
   \end{array}\right]\right\rangle \m\leq 0 \m,$$
which implies that
\begin{equation} \label{passivity:up and yp}
   \Re\m \langle u^{out}_p, y_p \rangle \m\geq\m 0 \m.
\end{equation}
Since $\Sig_c$ with the structure \rfb{minreal} is impedance passive
even if we set $d=0$, by the just mentioned result of Staffans we have
$\Re\sbm{a & b\\ -c & 0}\leq 0$, which implies
\begin{equation} \label{ineq:abcd d}
   \Re\m \sbm{ a & b \\[3pt] -c & -d } \m\leq\m \sbm{ 0 & 0 \\[3pt]
   0 & -\Re\m d} \m.
\end{equation}
From \rfb{eq:eigenAB} and \rfb{eq:upUyp}, we get
$$ \sbm{ a & b \\[3pt] -c & -d } \sbm{ q_p \\[3pt] -y_p } \m=\m
   \sbm{ i \o q_p \\[3pt] -u^{out}_p } \m.$$
Hence
$$ \begin{aligned} \Re <u^{out}_p, y_p> & = \Re \left[i \o\left\|q_p
   \right\|^2+<y_p, u^{out}_p>\right] \\ & = \Re \left\langle
   \sbm{ q_p \\[3pt] -y_p }, \sbm{ i\o q_p \\[3pt] -u^{out}_p }
   \right\rangle \\ & = \Re \left\langle \sbm{ q_p \\[3pt] -y_p },
   \sbm{ a & b \\[3pt] -c & -d } \sbm{ q_p \\[3pt] -y_p }
   \right\rangle . \end{aligned}$$
Using \rfb{ineq:abcd d}, we get
$$ \Re \langle u^{out}_p, y_p \rangle \m\leq\m -(\Re\m d)
   \|y_p\|^2 \m.$$
Comparing this with \rfb{passivity:up and yp} we get
$$ \Re \langle u^{out}_p, y_p \rangle \m=\m 0 \m.$$
Since $\Re d>0$ by assumption, we get $y_p=0$.  Using this fact
together with \rfb{eq:eigenAB} and \rfb{eq:upUyp}, we get
\begin{equation} \label{eq:charac_contr}
   a q_p \m=\m i \o q_p, \ \qquad c q_p \m=\m u^{out}_p,
\end{equation}
\begin{equation} \label{eq:charac_plant}
   (i\o I-A) x_p \m=\m B u^{out}_p, \ \quad [C \& D] \sbm{ x_p \\[3pt]
   u^{out}_p } \m=\m 0 \m.
\end{equation}

Since $i\o\notin\sigma(a)$, we obtain from \rfb{eq:charac_contr} that
$q_p=0$, hence $u_p^{out}=0$. Since $\sbm{x_p\\ q_p}\not=0$, we must
have $x_p\not=0$. Since $[\CD]\sbm{x_p\\ 0}=Cx_p$, we see from
\rfb{eq:charac_plant} that $x_p$ is an unobservable eigenvector of $A$
corresponding to the eigenvalue $i\o$. According to {the
assumption on $\sigma_p(A)\bigcap i\rline$} in Theorem
\ref{theorem:stab}, this is impossible.

{\bf Third case of $i\o\in\sigma_p(A^K)$} is when $\o=\o_j$ for some
$j\in\Jscr$. We have $i\o\notin\sigma(A)$ by assumption. Notice that
all the derivations we did in the second case are still valid, up to
and including \rfb{eq:charac_plant}. Thus, from \rfb{eq:charac_plant}
we get
$$ x_p \m=\m (i\o I-A)^{-1} B u^{out}_p \m.$$
Substitute this into the second equation of \rfb{eq:charac_plant}, to
get
$$ [C \& D] \sbm{(i\o I-A)^{-1} B \\ I} u^{out}_p \m=\m \PPP_u
   (i\o) u^{out}_p \m=\m 0 \m.$$
According to the assumption in Theorem \ref{theorem:stab}, we have
that $\PPP_u(i\o_j)$ is invertible, for all $j\in\Jscr$, and thus
$u^{out}_p=0$. Going back to \rfb{eq:charac_plant}, we get
$$ A x_p \m=\m i\o x_p \m.$$
Thus $x_p$ has to be zero, otherwise $i\o_j$ would be an eigenvalue
of $A$, which is a contradiction to our assumption.

Since $\Sig_c$ is minimal, from \rfb{eq:charac_contr} and by the fact
that $u_p^{out}=0$, it follows that $q_p=0$. Thus $\sbm{x_p\\ q_p}=0$,
a contradiction. This finishes the third case, and so we have shown
that $A^K$ has no imaginary eigenvalues.

{\bf Step 2.} \m Now we prove that \rfb{Maddalen} holds. For this, it
will be enough to show that any point $i\o\in i\rline\setminus\sigma
(A)$ is not in $\sigma\left(A^K\right)$. We assume the opposite, that
$i\o\in\sigma\left(A^K\right)$, and show that this leads to a
contradiction. We have to consider two cases:

{\bf Case 1:} $\o\not\in\{\o_j\ |\ j\in\Jscr\}$. Notice that
$\sigma(A^o)=\sigma(A)\cup\sigma(a)$, hence $i\o\not\in\sigma(A^o)$.
According to \cite[Theorem 2.6]{art45}, $i\o$ is a Fredholm eigenvalue
of $A^K$ {(see \cite[Definition 2.1]{art45})}. But we have proved
earlier that $A^K$ has no eigenvalues on $i\rline$, which is a
contradiction.

{\bf Case 2:} $\o=\o_j$ for some $j\in\Jscr$. Since $\Sig_K$ is
impedance passive (as proved in the previous theorem), its semigroup
$\tline^K$ is a contraction semigroup on $X\times\cline^n$, so that
the right half-plane $\cline_0$ is in $\rho(A^K)$. Thus, $i\o$ is in
the boundary of $\sigma(A^K)$. According to { \cite[Proposition
1.10, IV]{EngNag},} $i\o$ is in the approximate point spectrum of
$A^K$. Thus, according to the definition of this type of spectrum, it
is either an eigenvalue of $A^K$ (which we know that it cannot be true)
or the range Ran $(i\o I-A^K)$ is not closed. Hence, $i\o\in\sigma_e
(A^K)$, the essential spectrum of $A^K$, as defined for instance in
\cite[Sect.~2]{art45}. According to \cite[Proposition 2.5]{art45},
$\sigma_e(A^o)=\sigma_e(A^K)$ (the essential spectrum is invariant
under compact feedback). Thus, $i\o\in\sigma_e(A^o)$. According to
\cite[Proposition 2.3]{art45}, $i\o$ is not a Fredholm eigenvalue of
$A^o$. But on the other hand, from \rfb{Oshrit} we see that $i\o$,
being a Fredholm eigenvalue of $a$, is a Fredholm eigenvalue of $A^o$,
which is a contradiction. Thus we have proved \rfb{Maddalen}.

{\bf Step 3.} \m We show that $S(i\o_j)=0$. We have just proved (in
Case 2 above) that $i\o_j\in\rho(A^K)$. Since the transfer function of
any well-posed system has an analytic extension to the resolvent set
of its generator (see, e.g., \cite{art10}), the function $\GGG^K$ from
\rfb{tf:input_to_output} is analytic in a neighborhood of $i\o_j$. The
right upper block of $\GGG^K$ is $I-S$, so that also $S$ is analytic
around $i\o$. Now it follows from \rfb{Maddalen_to_Bordeaux} that
indeed $S(i\o_j)=0$.

{\bf Step 4.} \m It remains to prove the strong stability of
$\tline^K$. According to the well-known Arendt-Batty Theorem
\cite{AreBat} (see also \cite{PaZw2013}), if $A^K$ has no eigenvalues
on $i\rline$ (which we have already proved) and $\sigma\left(A^K
\right)\cap i\rline$ is countable, then $\tline^K$ is strongly stable.
The countable property follows immediately from \rfb{Maddalen} and the
assumption that $\sigma(A)\cap i\rline$ is countable. Therefore,
$\tline^K$ is strongly stable. \hfill $\square$ \end{pf}

Now we consider the output regulator problem for impedance passive
linear systems as in Fig.~\ref{fig:blocks}, with a disturbance input
$w$ and a control input $u$. {We deote by $B$ and $B_w$ the
control operators from the inputs $u$ and $w$, so that the state
trajectories of \m $\Sig_p$ are the solutions of
$$ \dot x(t) \m=\m A x(t) + B u(t) + B_w w(t) \m,$$
where $B$ and $B_w$ are admissible control operators for the semigroup
$\tline$ generated by $A$.}

\begin{theorem} \label{thm:main}
Suppose that $\Sig_p$ is a well-posed linear system with transfer
function $\PPP=[\PPP_w\ \PPP_u]$, $\Sig_p$ is impedance passive from
$u$ to $y$ and satisfies Assumption \ref{ass_trans}. Let the
controller $\Sig_c$ be an impedance passive and minimal realization of
the transfer function $\gg$ from \rfb{tf_controller}, with state space
$X_c$. Then the closed-loop system $\Sig_L$ in Fig. \ref{fig:blocks}
with input $\sbm{w\\ v\\ r}$ and output $\sbm{y\\ u^{out}}$, with
state space $X\times X_c$, is well-posed. This system is input-output
stable and impedance passive from $\sbm{v \\ r}$ to
$\sbm{y \\ u^{out}}$.

Assume additionally that $\Ker\m B=\{0\}$, $\sigma_p(A)\bigcap i
\rline$ consists only of observable eigenvalues, and $\sigma(A)\bigcap
i\rline$ is countable. Then the closed-loop system $\Sig_L$ is
strongly stable. If $w,\m v$ and $r$ are of the form \rfb{exosig},
then $\Sig_c$ solves the output regulator problem (as defined in
Section 1). {Moreover, the signal $t\mapsto\|e(t)\|^2$ (and hence
also $t\mapsto\|e(t)\|$) tends to zero when filtered.}
\end{theorem}

\vspace{-2mm}
\begin{pf} \m Similarly as in the proof of Theorem
\ref{theorem:well}, we can obtain the system $\Sig_L$ in Fig.
\ref{fig:blocks} from the ``open{-}loop system'' consisting of
$\Sig_p$ and $\Sig_c$ by the feedback operator $L=\sbm{0 & 0\\ I & 0
\\ 0 & -I}$ acting on the input $\sbm{w \\ u \\ e}$ from the output
$\sbm{y \\ u^{out}}$. The state space of $\Sig_L$ is $X\times X_c$.
The transfer function of the system $\Sig_L$ from the input
$\sbm{w \\ v \\ r}$ to the output $\sbm{y\\ u^{out}}$ is
\vspace{-2mm}
\begin{equation} \label{tf:input_to_output_1}
  \GGG^L \m=\m \bbm{ \begin{matrix} S\m\PPP_w \\ -\gg \m S\m \PPP_w
  \end{matrix} & \vline height 6mm depth 4mm & \m\ \GGG^K \ \m } ,
\end{equation}
where $\GGG^K$ is the transfer function from $\sbm{v \\ r}$ to
$\sbm{y \\ u^{out}}$ given by \rfb{tf:input_to_output}. We know from
Theorem \ref{theorem:well} that $\GGG^K,\m S\in H^\infty_0$. Since
$\PPP_w$ is a proper transfer function, i.e., $\PPP_w\in H^\infty_m$
for some $m\in\rline$, clearly $\GGG^L$ is proper and hence the
feedback system $\Sig_L$ is well-posed, see \cite{art10}. The
input-output stability and passivity properties follow easily from
Theorem \ref{theorem:well}. The first part is proved.

If $w=0$ then the system $\Sig_L$ reduces to the system $\Sig_K$ from
the previous two theorems, so that the semigroup of \m $\Sig_L$ is the
same as the semigroup $\tline_K$ of \m $\Sig_K$. Hence, under the
additional assumptions stated in the second part of the theorem, it
follows from Theorem \ref{theorem:stab} that $\Sig_L$ is strongly
stable.

It remains to show that with the inputs from \rfb{exosig}, any state
trajectory of \m $\Sig_L$ is bounded in $X\times X_c$ and $\|e(t)\|$
is convergent to zero when filtered. Thanks to the superposition
principle, it is enough to prove these statements for a single
frequency $\o_j$ (where $j\in\Jscr$), i.e., assuming that there is
only one term present in each of the sums \rfb{exosig}. Then every
state trajectory of $\Sig_L$ is a strong solution of
\begin{equation} \label{Adir}
   \dot z(t) \m=\m A^K z(t) + B^L\sbm{w(t)\\ v(t)\\ r(t)} ,
\end{equation}
where $B^L$ is the control operator of $\Sig_L$ and $w(t)=w_j
e^{i\o_j t}$, $v(t)=v_j e^{i\o_j t}$, $r(t)=r_j e^{i\o_j t}$. For the
concept of strong solution of an equation of the type \rfb{Adir} we
refer, for instance, to \cite[Sect.~2]{art11}. Essentially, it means
that if we integrate both sides in a suitable extrapolation space
from 0 to $t$, then (for every $t\geq 0$) we get an equality in the
state space $X\times X_c$. {It can be verified by substitution
into \eqref{Adir} that a possible state trajectory of $\Sig_L$ is}
\begin{equation} \label{SS}
   z_{ss}(t) \m=\m (i\o_j I-A^K)^{-1}B^L \sbm{w(t)\\ v(t)\\ r(t)}
\end{equation}
(this is called the {\em steady-state} state trajectory). The above
expression makes sense, because we know from Theorem
\ref{theorem:stab} that $i\o_j\in\rho(A^K)$. Denote $z_{tr}=z-z_{ss}$
(this is called the {\em transient} state trajectory), then it is easy
to see that $z_{tr}$ is a strong solution of $\dot z_{tr}(t)=A^K
z_{tr}(t)$. Since $A^K$ is strongly stable, we have \m $\lim_{t\rarrow
\infty}z_{tr}(t)=0$ (so that for large $t$ we have $z(t)\approx z_{ss}
(t)$). In particular, it is now clear that any state trajectory $z$ is
bounded, as required in the output regulator problem.

The alternative output function $e=r-y$ of $\Sig_L$ can be
decomposed as $e=e_{ss}+e_{tr}$, where $e_{ss}$ is the component
corresponding to $z_{ss}$ and $e_{tr}$ is the component corresponding
to $z_{tr}$:
$$ e_{ss} = \Psi^L_\infty z_{ss}(0) +\fline^L_\infty \sbm{w\\
   v\\ r},\ \ \ e_{tr} = \Psi^L_\infty z_{tr}(0),$$
where $\Psi^L_\infty$ is the extended output map of $\Sig_L$ for the
output $e$, as introduced after \rfb{Jeremy_Bowen} and $\fline^L
_\infty$ is the extended input-output map of $\Sig_L$ for the output
$e$, see \cite[Sect.~2]{art10} for the concepts. {According to
Theorem \ref{Macron}, $\|e_{tr}(t)\|^2$ and $\|e_{tr}(t)\|$ tend to
$0$ when filtered.}

{
Now we prove that $e_{ss}=0$. We denote by $[\CD]^{Le}$ the combined
observation/feedthrough operator of \m $\Sig_L$ for the output $e$,
see \cite[Sect.~3]{art11} for the background on this concept.} It is
clear from \rfb{SS} that $[z_{ss}(t)\ w(t)\ v(t)\ r(t)]^\top\in\Dscr
([\CD]^{Le})$ for all $t\geq 0$. It follows from \cite[Theorem 3.1]
{art11} that for all $t\geq 0$,
$$ e_{ss}(t) \m=\m [\CD]^{Le} \sbm{ z_{ss}(t)\\ w(t)\\ v(t)\\ r(t)}
   \ \ \ \mbox{(now use \rfb{SS})}$$
$$ = [\CD]^{Le} \sbm{ (i\o_j I-A^K)^{-1}B^L\\[2pt] I} \nm \sbm{w(t)\\
   v(t)\\ r(t)} \nm= \GGG^{Le}(i\o_j) \sbm{w(t)\\ v(t)\\ r(t)} \nm,$$
where $\GGG^{Le}$ is the transfer function of $\Sig_L$ from $\sbm{w\\
v\\ r}$ to $e$. In the last step, we have used
\cite[eq.~(3.4)]{art11}. From \rfb{tf:input_to_output_1} and
\rfb{tf:input_to_output}, the transfer function $\GGG^{Le}$ is
$$ \GGG^{Le} \m=\m \bbm{ -S\PPP_w \ \m & -S\PPP_u \ \m & S}.$$
From the fact that $S(i\o_j)=0$ (Theorem \ref{theorem:stab}) it now
follows that $e_{ss}=0$. {Hence, $\|e(t)\|^2\rarrow 0$ when
filtered.} \hfill $\square$ \end{pf}

\begin{remark} \label{Remark-bounded-obserCop}
Under the conditions of Theorem \ref{thm:main}, if the observation
operator $C$ of \m $\Sigma_p$ is bounded from $X$ to $U$, then we
obtain a stronger result: the norm of the tracking error, $\|e(t)\|$
converges to zero. To see this fact, we use the decomposition
$e=e_{ss}+e_{tr}$ as in the proof of Theorem \ref{thm:main}. It
follows from the boundedness of $C$ and the strong stability of the
feedback system $\Sig_L$ that $\|e_{tr}(t)\|$ converges to
zero. Hence, $\|e(t)\|\rarrow 0$.
\end{remark}

\begin{remark} {
If we assume in Theorem \ref{thm:main} one more condition, that the
plant $\Sig_p$ is a regular system, then the minimal controller
from \rfb{minreal} solves the {\em robust regulation problem}, as
defined in \cite{Paun2017rob}.} Indeed, it is easy to see from
\rfb{minreal} that the controller satisfies the so-called
$\Gscr$-conditions introduced in \cite{Paun2017rob}:
\begin{align*}
   & \Ran\, (i\o_j I-a)\bigcap \Ran\, b \m=\m \{0\} \FORALL
   j \in \Jscr, \\ & \Ker\; b \m=\m \{0\} \m.
\end{align*}
Moreover, since $i\o_j\in\rho(A^K)$, according to
\cite[Theorem 3.8]{Paun2017rob}, this controller solves the robust
regulation problem. In particular, the tracking error satisfies
$$ \int_t^{t+1} \|e(t)\|\dd t \m\rarrow\m 0\ \mbox{ as }\
   t \m\rarrow\m \infty,$$
from which it follows (using Proposition \ref{Fordo}) that
$\|e(t)\|$ tends to zero when filtered.
The robust regulation in the sense of \cite{Paun2017rob}
considers only perturbed closed-loop systems that remain regular and
strongly stable. \m
\end{remark}

\begin{remark} Theorems \ref{theorem:well}-\ref{thm:main} show
that the design of the proposed controller $\Sig_c$ requires little
information about the plant. $\Sig_c$ is robust in the following
sense: Consider a perturbed system $\tilde\Sig_p$ that is still
well-posed and impedance passive. If its transfer function $\tilde{P}$
still satisfies Assumption 1.1 (invertibility of $\tilde{P}(i\o_j)$),
then $\Sig_c$ will still achieve an impedance passive and input-output
stable closed-loop system. If the perturbed system satisfies the
structural assumptions of Theorem \ref{thm:main} ($\Ker\m\tilde{B}=\{
0\}$, $\sigma(\tilde{A})\cap i\rline$ is countable, and $\sigma_p
(\tilde A)\cap i\rline$ consists only of observable eigenvalues),
then $\Sig_c$ solves the output regulator problem for $\tilde\Sig_p$.
\m \end{remark}

\begin{remark}
We feel that it is difficult to find generally valid guidelines
on the parameter tuning of the controller \rfb{tf_controller}.
According to well known root locus principles, for small $c_j>0$,
the closed-loop eigenvalue trajectory starting at $i\o_j$ will move
away from $i\o_j$, so increasing $c_j$ at the beginning improves the
real part of this eigenvalue (by pushing it to the left). However, as
we increase $c_j$ to large values, the eigenvalue will converge to a
zero of the loop gain $\PPP_u\cdot\gg$, or to $\infty$. The zeros of
\m $\PPP_u$ may be on the imaginary axis, so that increasing $c_j$ too
much may bring the eigenvalue back, closer to the imaginary axis.
Without further research on this topic, we can only recommend
empirical tuning based on simulations. \end{remark}

\section{Examples} \setcounter{equation}{0} \label{sec4} 

\begin{example} {{Strong stabilization of Timoshenko beam by
an internal model controller.}}
\end{example}
{
Consider a Timoshenko beam on spatial domain $[0,1]$ with transverse
displacement $w(x,t)$ and rotation angle $\phi(x,t)$. The dynamics are
governed by the coupled partial differential equations
\begin{equation} \label{TimoBeam}
   \begin{aligned} \rho(x) w_{tt}(x,t) =\; &\partial_x [ K(x)(w_x(x,t)
   - \phi(x,t) ) ], \\ I_\rho(x) \phi_{tt}(x,t) =\; & \partial_x
   [ EI(x) \phi_x(x,t)  ] \\ &+ K(x)(w_x (x,t) - \phi(x,t) ),
   \end{aligned}
\end{equation}
where mass density $\rho$, moment of inertia $I_\rho$, bending
stiffness $EI$, and shear stiffness $K$ are spatially varying
parameters. We assume that $\rho,I_\rho,EI,K\in C^1([0,1];\rline)$
and they are strictly positive. The system is clamped at $x=0$,
meaning that
$$ w(0,t) \m=\m 0 \m,\qquad \phi(0,t) \m=\m 0 \m.$$
At the free end $x=1$, the  control \textit{input} $u(t)$ and
collocated \textit{output} $y(t)$ are defined as:
\begin{equation} \label{Beam:boundary}
   u(t) \m= \begin{bmatrix} K(1)(w_x(1,t) - \phi(1,t)) \\ EI(1)
   \phi_x(1,t) \end{bmatrix}, \quad y(t) = \begin{bmatrix} w_t(1,t) \\
   \phi_t(1,t) \end{bmatrix}.
\end{equation}
For the background on beam equations we refer to the survey
\cite{Han_on_beams} and in particular, for 
{Timoshenko} beams modelled
in the port-Hamiltonian context, to
\cite{Jacob_Zwart,BirHan,Machelli,paunonenYanRam2021}.}

{
We interconnect this plant $\Sig_p$ with an \textbf{IMC}
with output $u^{out}$, and input $e=-y$ (meaning that $r=0$), as
shown in Fig.~\ref{fig:2}. Assume that $u=u^{out}+v$, where $v$ is
an input disturbance to the plant.}

{
We cast the plant into a port-Hamiltonian framework. Define the state
vector $z = [w_x - \phi, \phi_x, \rho w_t, I_\rho \phi_t]^T$ in space
$X=(L^2[0,1])^4$. The energy inner product is given by $\langle z,
\phi \rangle_X=\int_0^1 z^\top H(x)\phi\,\dd x$, where
$$H(x) \m=\m {\rm diag}(K(x), EI(x), 1/\rho(x), 1/I_\rho(x))$$
is a strictly positive energy density matrix.}

{The co-energy variables are $\xi=Hz$. The system $\Sig_p$ in
\rfb{TimoBeam}-\rfb{Beam:boundary} can be written in the
abstract form
\begin{equation}\label{abs:beam}
    \dot{z}(t) = Az(t) + Bu(t), \ \ \ y(t) = Cz(t),
\end{equation}
where the system operator $A z=P_1 \partial_x \xi + P_0 \xi$ with
\begin{equation*}
   \begin{aligned} & \Dscr(A) \m=\m \left\{ z \in X \mid \xi \in
   (\Hscr^1(0,1))^4,\ \xi = Hz, \right. \\ & \qquad \qquad \left.
   \ \xi_{3,4}(0)=0, \ \xi_{1,2}(1)=0 \right\},\\ & P_1 \m=
   \begin{bmatrix} 0 & I_2 \\ I_2 & 0 \end{bmatrix}, \ P_0 \m=
   \begin{bmatrix} 0 & -N \\ N^T & 0 \end{bmatrix},\ N \m=
   \begin{bmatrix} 0 & 1 \\ 0 & 0 \end{bmatrix}, \\ & B = \delta(x-1)
   \begin{bmatrix} 0_{2 \times 2} \\ I_2 \end{bmatrix}, \end{aligned}
\end{equation*}
and $C z = [\xi_3(1)\ \xi_4(1)]^\top$. It is easy to check that $A$
is skew-adjoint and has compact resolvents.}

{
\begin{proposition}
The system $\Sig_p$ from \rfb{abs:beam} is well-posed and it is
impedance passive from $u$ to $y$.
\end{proposition}}

\begin{pf} {
We now use Theorem \cite[Theorem 5.2]{BirHan} to obtain the
well-posedness of \eqref{abs:beam}. For that purpose, we express the
boundary variables using the boundary flow $f_\partial$ and effort
$e_\partial$:
\begin{equation}
   \begin{bmatrix} f_\partial \\ e_\partial \end{bmatrix} \m=\m
   \frac{1}{\sqrt{2}} \begin{bmatrix} P_1 & -P_1 \\ I_4 & I_4
   \end{bmatrix} \begin{bmatrix} \xi(1) \\ \xi(0) \end{bmatrix}.
\end{equation}
The boundary constraints $u(t) = [\xi_1(1), \xi_2(1)]^\top$ and
$[\xi_3(0), \xi_4(0)]^\top = 0$ can be reformulated as $W_B
[\begin{smallmatrix} f_\partial \\ e_\partial \end{smallmatrix}] =
[\begin{smallmatrix} u(t) \\ 0 \end{smallmatrix}]$, where the
boundary matrix is:
\begin{equation}
   W_B \m=\m \frac{1}{\sqrt{2}} \begin{bmatrix} 0 & I_2 & I_2 & 0 \\
   -I_2 & 0 & 0 & I_2 \end{bmatrix}.
\end{equation}
With $\Upsilon=\left[\begin{smallmatrix} 0 & I_4 \\ I_4 & 0
\end{smallmatrix}\right]$, a direct algebraic computation verifies
that
\begin{equation}
   W_B \Upsilon W_B^\top \m=\m 0.
\end{equation}
Let $c_s=\sqrt{K/\rho}$, $c_b = \sqrt{EI/I_\rho}$, $Z_s = \sqrt{\rho
K}$ and $Z_b = \sqrt{I_\rho EI}$. We have $P_1 H(x) = \bar{S}^{-1}
(x)\Delta(x) \bar{S}(x)$, where the eigenvalue matrix is $\Delta =
\operatorname{diag}(c_s, -c_s, c_b, -c_b)$, and
\begin{equation}
   \bar{S} = \m\half \left[\begin{smallmatrix} 1 & 0 & Z_s^{-1} & 0
   \\ 1 & 0 & -Z_s^{-1} & 0 \\ 0 & 1 & 0 & Z_b^{-1} \\ 0 & 1 & 0 &
   -Z_b^{-1} \end{smallmatrix}\right], \quad
   \bar{S}^{-1} = \left[\begin{smallmatrix} 1 & 1 & 0 & 0 \\
   0 & 0 & 1 & 1 \\ Z_s & -Z_s & 0 & 0 \\ 0 & 0 & Z_b & -Z_b
   \end{smallmatrix}\right].
\end{equation}
Clearly, $\bar{S}, \bar{S}^{-1}$ and $\Delta$ are continuously
differentiable. By \cite[Theorem 5.2]{BirHan}, we obtain the
conclusion. A formal calculation using \cite[formula (5)]{BirHan}
shows that $\Sig_p$ is impedance passive from $u$ to $y$.
\hfill $\square$} \end{pf} \vspace{-5mm}

{
\begin{proposition}\label{pro:obs}
The pair $(A,C)$ has no unobservable eigenvalues.
\end{proposition}} \vspace{-2mm}

\begin{pf} {
Suppose that there exists an unobservable eigenvalue $p\in\cline$
with a non-trivial eigenvector $z\in\Dscr(A)$ ($z\neq 0$). The
conditions $Az=pz$ and $Cz=0$ imply that $P_1\xi'+P_0\xi=p H^{-1}
(x)\xi$. Since $P_1^{-1} = P_1$, isolating the derivative yields
\begin{equation} \label{eq:ode}
    \xi'(x) =\m P_1\big(p H^{-1}(x) - P_0\big)\xi(x) := M(x)\xi(x).
\end{equation}
Because $H^{-1}(x)$ is continuously differentiable, the coefficient
matrix $M(x)$ is continuous on $[0,1]$. The domain constraint $z\in
\Dscr(A)$ yields $\xi_{1}(1)=\xi_{2}(1)=0$, and the unobservability
condition $Cz = 0$ yields $\xi_3(1)=\xi_4(1)=0$. Together, these
imply that $\xi(1) = {0}$. According to the Picard-Lindel\"of
theorem, the unique solution to the linear homogeneous ODE
\rfb{eq:ode} with a trivial terminal condition is $\xi(x)\equiv
{0}$ for all $x\in [0,1]$. Given that $H(x)$ is strictly positive
definite, $\xi(x)\equiv 0$ implies $z(x)=H^{-1}(x)\xi(x)\equiv 0$.
This contradicts the assumption that $z$ is a non-trivial
eigenvector. \hfill $\square$} \end{pf}

{
Using the above proposition and Theorems \ref{theorem:well},
\ref{thm:main}, we can derive our main result about this example:}

\vspace{1mm}
\begin{theorem} \label{theorem:SCOLE} {
Let $\Sig_p$ be the well-posed system corresponding to the Timoshenko
beam model \rfb{TimoBeam}-\rfb{Beam:boundary}, with transfer function
$[\PPP_w\ \PPP_u]$ and let $\Sigma_c$ be a minimal \textbf{IMC} as in
\rfb{minreal}, with resonant frequencies \m $\o_1, ...\,\o_n$ that
are not among the zeros of $\PPP_u$ and also not among the
eigenvalues of $A$. Let $\Sig_L$ be the closed-loop system obtained
when $\Sig_p$ and $\Sig_c$ are interconnected as in Fig.~\ref{fig:2},
with state space is $X^K=X\times\cline^l$\m, where $X$ is the state
space of $\Sig_p$ and $\cline^l$ is the state space of $\Sig_c$.}

{
Then $\Sig_L$ is well-posed with input $v$, state $\sbm{x\\ q}\in
X^K$ and output $\sbm{y\\ u^{out}}$, where $x$ is the state of \m
$\Sig_p$ and $q$ is the state of \m $\Sig_c$. Moreover, if $v$ is as
in \rfb{exosig} and $r=0$, then the signal $t\mapsto \|y(t)\|^2$
tends to zero when filtered.}
\end{theorem}

\begin{pf} \m 
{We show the well-posedness result first. Since $i\o_j$ is not
among the zeros nor among the eigenvalues of $\Sig_p$, by the
definition of zeros for $\Sig_p$, see for instance \cite[Definition
V.1]{art66}, we have that ${\rm rank}(\PPP_u(i\o_j))=2$. Hence, the
conditions in Assumption \ref{ass_trans} are satisfied. Then, from
Theorem \ref{thm:main} we know that the coupled system $\Sig_L$ is
well-posed on $X^K$ with input $v$, state $\sbm{x\\ q}\in X^K$ and
output $\sbm{y\\ u^{out}}$.}

{Now we check the remaining conditions for stability. Clearly
$\Ker\m B=\{0\}$. Since $A$ has compact resolvents, we know that
$\sigma(A)\cap i\rline$ is countable. From Proposition \ref{pro:obs},
we know that $\sigma(A)\cap i\rline$ consists only of observable
eigenvalues. Thus $\Sig_L$ is strongly stable on $X^K$. The
conclusion then follows directly from Theorem \ref{thm:main}.
\hfill $\square$} \end{pf} \vspace{-4mm}

\textbf{Simulation setup:} \m {Numerical simulations were
implemented in MATLAB, using the finite difference method for spatial
discretization and the stiff solver \texttt{ode15s} for time
integration. This numerical scheme is also adopted in the subsequent
example. In \dref{TimoBeam}, we set $K(x)=\rho(x)= EI(x)=I_\rho(x)=1$
for $x\in[0,1]$. We only use the torque control $u_2$ and the shear
force control $u_1$ is not implemented, i.e., $u_1(t)=0$. The output
signal is only the angular velocity $y_2$. The disturbance $v_2(t)=
\sin(3t)$. The parameters of the transfer function
\rfb{tf_controller} are selected as $c_1=1$, $\o_1=3$, and $d=5$. One
impedance passive and minimal realization $\Sig_c$ of the transfer
function \rfb{tf_controller} is
\begin{equation}\label{controller:TimoBeam}
   \begin{aligned}
   & \dot{q}(t) \m=\m \bbm{0 & -3 \\ 3 & 0} q(t)+\bbm{1 \\ 1} e(t),\\
   & u^{out}(t) \m=\m \bbm{1 & 1} q(t) + 5 e(t),
\end{aligned}
\end{equation}
which is connected to the plant as shown in Fig.~\ref{fig:2}. }

\noindent {\bf Result:} {
A computation shows that $|\PPP_u(\pm 3i)|\approx 0.27$, so that
$\PPP_u(\pm 3i)$ is invertible. The closed-loop system response in
Fig.~\ref{fig:Pu_1} is consistent with the asymptotic stability
of \rfb{TimoBeam} with the controller $\Sig_c$.}

\begin{figure}[htbp] \centering 
   \includegraphics[width=1\linewidth]{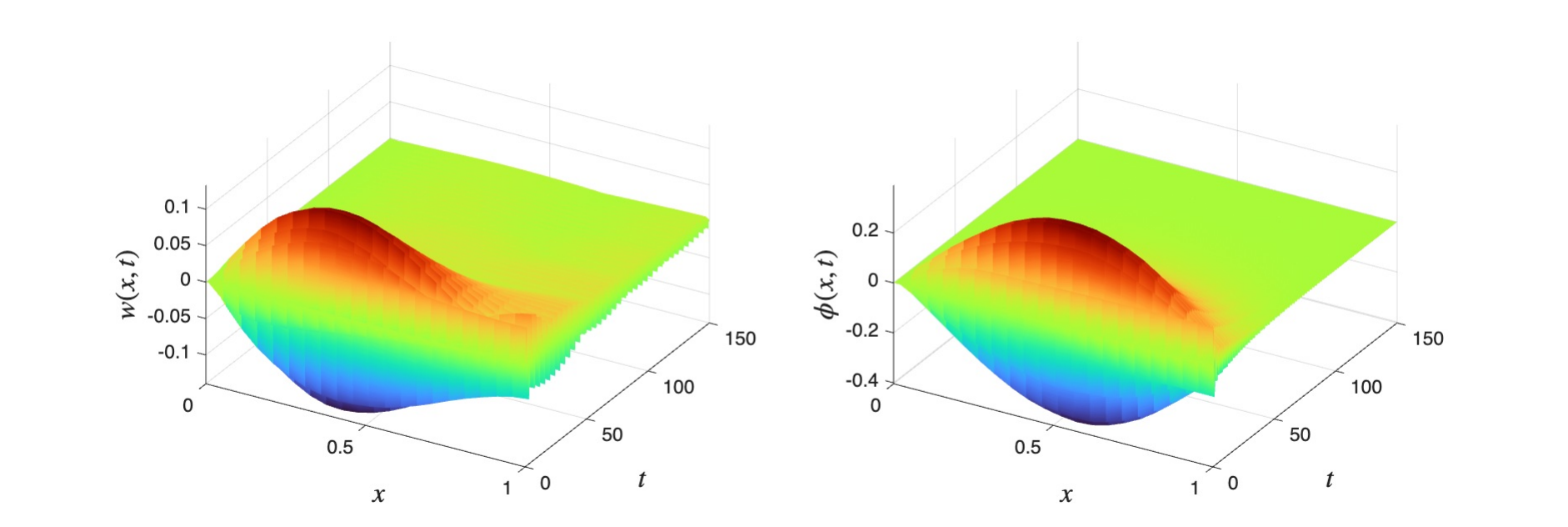}
   \caption{Plots of the evolution of the system \rfb{TimoBeam}
     with the feedback controller \rfb{tf_controller}}
   \label{fig:Pu_1}
\end{figure}



\begin{example} {Output regulation for a voltage-actuated
piezoelectric beam with magnetic effects.}
\end{example}

{
The plant $\Sig_p$ represents the stretching dynamics of a voltage
actuated piezoelectric beam with magnetic effects, as modeled in
\cite{morris2014modeling}. (For the stabilizability of a current
controlled piezoelectric beam we refer to \cite{Ozer_Morris_2020}
and for a 3 dimensional extension of the model from
\cite{morris2014modeling} we refer to \cite{LiaoFeng}.) $\Sig_p$ is
governed by
$$ \rho z_{tt}(x,t)-\alpha z_{xx}(x,t) + \gamma \beta p_{xx}(x,t)
   + f_1(x)w_1(t) \m=\m 0,$$
$$ \mu p_{tt}{(x,t)}-\beta p_{xx}(x,t) + \gamma \beta z_{xx}(x,t)
   + f_2(x)w_2(t) \m=\m 0,$$
$$ z(0,t)=p(0,t)=0,\ \ \alpha z_x(1,t)-\gamma \beta p_x(1,t)
   \m=\m 0,$$
\begin{equation} \label{eq:piezobeam}
   \beta p_x(1,t)-\gamma \beta z_x(1,t) \m=\m \frac{u(t)}{h} \m,
\end{equation} \vspace{-4mm}
$$ y(t) \m=\m \frac{1}{h} \m \dot{p}(1,t),$$
where $z(x,t)$ is the longitudinal displacement of the beam, $p(x,t)$
is the total electric charge accumulated from position 0 up to
position $x$ along the piezoelectric beam at time $t$, $h$ is the beam
thickness, $u(t)$ is the {\em control input}, the voltage applied
across the electrodes of the piezoelectric beam, $\mu$ is the magnetic
permeability of the beam, $\rho$ is the mass density per unit volume,
$\alpha$ is the elastic stiffness, $\gamma$ is the piezoelectric
coefficient, and $\beta$ is the impermittivity coefficient (the
inverse of the dielectric permittivity). It is assumed that
$$ \alpha_1 \m=\m \alpha - \gamma^2\beta \m>\m 0 \m.$$
The signal $y$ is the {\em output}, where $\dot{p}(1,t)$ is the total
current flowing through the electrodes. $w_i, i=1,2$ are the external
disturbances of the form \rfb{d1} and $f_1,f_2\in\Hscr^1(0,1)$ are
coefficient functions which are unknown.} The reference signal $r$
to be tracked is assumed to have the form \rfb{d2}.

The state space of $\Sig_p$ is
\begin{equation} \label{statespace2}
   X \m=\m (\Hscr_L^1(0,1))^2\times (L^2[0,1])^2,
\end{equation}
where $\Hscr_L^1(0,1)=\{z\in\Hscr^1(0,1)\,|\, z(0)=0\}$. The
equations \rfb{eq:piezobeam} can be rewritten as
\begin{equation}\label{eq:piebeam_state}
   \begin{cases} \dot{\psi} (t) \m=\m \Ascr \psi(t) + B u(t) +
   B_d w(t),\\ y(t) \m=\m B^* \psi(t), \end{cases}
\end{equation}
where the state of $\Sig_p$ is
\[ \begin{aligned}
   \psi(t) &= [\psi_1(t)\ \ \psi_2(t)\ \ \psi_3(t)\ \ \psi_4(t)]
   ^\top\\ &= [z(\cdot,t)\ \ p(\cdot,t)\ \ \dot{z}(\cdot,t)\ \
   \dot{p}(\cdot,t)]^\top , \end{aligned} \]
   and
$$ \Ascr\psi \m=\m \bbm{ \psi_3 \\[5pt] \psi_4 \\[5pt]
   \dfrac{\alpha}{\rho} \dfrac{\partial^2 \psi_1}
   {\partial x^2} - \dfrac{\gamma \beta}{\rho} \dfrac{\partial^2
   \psi_2}{\partial x^2} \\[10pt] -\dfrac{\gamma \beta}{\mu}
   \dfrac{\partial^2 \psi_1}{\partial x^2} + \dfrac{\beta}{\mu}
   \dfrac{\partial^2 \psi_2}{\partial x^2} },$$
$$ \Dscr(\Ascr) \m=\m \left\{ \psi \in X_0 \;\middle|\;
   \psi_{1x}(1) \m=\m \psi_{2x}(1) \m=\m 0 \right\},$$
$$ X_0 \m=\m (\Hscr^2(0,1) \cap \Hscr^1_L(0,1))^2 \times
   (\Hscr^1_L(0,1))^2,$$
$$ w(t) \m=\m [\, w_1(t) \;\; w_2(t) \,]^\top,$$
{ \vspace{-2mm}
$$ B \m= \bbm{ 0_{2\times1} \\ B_0}, \ \ B_0 \m= \bbm{0 \\
   \frac{1}{h}\delta_1 } ,\ \ B_d \m= \bbm{ 0 & 0 \\ 0 & 0 \\
   -f_1 & 0 \\ 0 & -f_2 } ,$$
\vspace{2mm}%
where $\delta_1$ is the functional of point evaluation at $x=1$.}

{
\begin{proposition} \label{prop:exa2_skew}
The operator $\Ascr$ is skew-adjoint on $X$ and has compact
resolvents. The system $\Sig_p$ from \rfb{eq:piebeam_state} is
well-posed and it is impedance passive from $u$ to $y$. (In
particular, $B$ is an admissible control operator for the semigroup
generated by $\Ascr$.)
\end{proposition}}

{ \vspace{-2mm}
\begin{pf} \m The skew adjointness and the compact resolvents follow
from \cite[Lemma 3.1]{morris2014modeling} and the well-posedness from
$u$ to $y$ follows from \cite[Theorem 3.5]{morris2014modeling}. The
control operator $B_d$ is bounded and thus the well-posedness from
$w$ to $y$ is not a problem. Starting from \rfb{eq:piebeam_state},
a formal calculation gives \dref{Tesla}, and thus $\Sig_p$ is
impedance passive from $u$ to $y$. \hfill $\square$ \m \end{pf}}

The transfer function $\PPP_u$ from $u$ to $y$ has been
computed (as a closed formula) in
\cite[Appendix, A.4]{morris2014modeling}.

Define the numbers $\zeta_1$ and $\zeta_2$ by
{\small \vspace{-\baselineskip} \[
   \zeta_{1,2} = \frac{1}{\sqrt{2}}\sqrt{\frac{\gamma^2\mu+\rho}
   {\alpha_1} + \frac{\mu}{\beta} \pm \sqrt{\left(
   \frac{\gamma^2\mu}{\alpha_1} + \frac{\mu}{\beta} +
   \frac{\rho}{\alpha_1}\right)^2 - \frac{4\rho\mu}
   {\beta\alpha_1}}}, \] }

\begin{proposition} \label{prop:obs2}
{If the condition
\begin{equation} \label{Hegseth}
   \frac{\zeta_1}{\zeta_2} \m\neq\m \frac{2n-1}{2m-1}, \ \
   \mbox{ for all } \ \ n,m\in\nline
\end{equation}
holds, then the spectrum $\sigma(\Ascr)$ consists only of
observable eigenvalues for the pair $(\Ascr,B^*)$.}
\end{proposition}

\vspace{-2mm}
\begin{pf} \m Define
$$ A_0 \m=\m \left[\begin{array}{ll} -\frac{\alpha}{\rho}\frac{\d^2}
   {\d x^2}\ & \ \frac{\gamma \beta}{\rho}\frac{\d^2}{\d x^2} \\
   \frac{\gamma \beta}{\mu}\frac{\d^2}{\d x^2}\ &\ -\frac{\beta}
   {\mu}\frac{\dd^2}{\dd x^2} \end{array}\right] \m,$$
{
where $\Dom{A_0}=\{\phi\in (\Hscr^2(0,1)\bigcap \Hscr^1_L(0,1))^2|\m
\phi_{1x}(1)=\phi_{2x}(1)=0 \}$. It is not difficult to verify that
$A_0>0$. The system \rfb{eq:piebeam_state}, with $w=0$, can be
rewritten as
\begin{equation} \label{eq:piebeam_state_b}
   \begin{cases} \dot{\psi}(t) \m=\m \left[\begin{array}{ll}
   0\ & \ I_{2\times 2} \\ -A_0\ & \ 0 \end{array}\right] \psi(t)
   + B u(t), \\ y(t) \m=\m B^* \psi(t). \end{cases}
\end{equation}
Consider the feedback law $u(t)=-\half y(t)+V(t)$, where $V$ is
a new input signal, which leads to the system
\begin{equation} \label{eq:piebeam_state_close}
   \begin{cases} \dot{\varphi}(t) \m=\m \Ascr_d \varphi(t) +
   B V(t),\\ y(t) \m=\m B^* \varphi(t), \end{cases}
\end{equation}
where \ $\Ascr_d = \begin{bmatrix} 0 & I_{2\times2} \\ -A_0 &
-\tfrac{1}{2}B_0B_0^* \end{bmatrix}$, \ and
$$ \Dom{\Ascr_d} \m=\m \Big\{ \varphi \in X_0\ \big|\ \alpha
   \varphi_{1x}(1) = \gamma\beta\varphi_{2x}(1),$$
\vspace{-5mm}
$$ \m\hspace{25mm}  \beta \varphi_{2x}(1)-\gamma \beta
   \varphi_{1x}(1) = -\frac{\varphi_4(1)}{2h^2} \Big\} \m.$$}

{
By \cite[Theorem 3.8]{morris2014modeling}, $\Ascr_d$ generates a
strongly stable semigroup if and only if \rfb{Hegseth} holds.
The system \rfb{eq:piebeam_state_close} can be rewritten in the
following form:
$$ \begin{cases}
   \begin{bmatrix} \ddot{z} \\ \ddot{p} \end{bmatrix} + A_0
   \begin{bmatrix} z \\ p \end{bmatrix} + \tfrac{1}{2} B_0 B_0^*
   \begin{bmatrix} \dot{z} \\ \dot{p} \end{bmatrix}
   \m=\m B_0 V(t), \\[6pt] y(t) \m=\m B_0^* \begin{bmatrix}
   \dot{z} \\ \dot{p} \end{bmatrix}, \end{cases}$$
which (except for a minor modification of the second equation) is in
the class of systems studied in \cite{art57}. We know from
\cite[Theorem 3.7]{morris2014modeling} that $\Ascr_d$ has compact
resolvents, hence its spectrum has no accumulation points. Now
according to \cite[Proposition 3.4]{art57}, $(\Ascr_d,B^*)$ is exactly
observable in infinite time if and only if $\Ascr_d$ generates a
strongly stable semigroup, which we know that it is the case if and
only if \rfb{Hegseth} holds. To conclude the proof, we have to show
that if $(\Ascr_d,B^*)$ is exactly observable in infinite time, then
all the eigenvalues of $\Ascr$ are observable for the pair
$(\Ascr,B^*)$.}

{
Suppose that $(\Ascr_d,B^*)$ is exactly observable in infinite time.
If it were true that $\Ascr$ has a nonobservable eigenvalue, then
there exists a corresponding eigenvector $\psi\in\Dscr(\Ascr)$ such
that $B^*\psi=0$. Then it is easy to see that $\Ascr_d\psi=\Ascr\psi$
so that $\psi$ is an eigenvector of $\Ascr_d$, contradicting that
$(\Ascr_d,B^*)$ is exactly observable in infinite time. Hence, all
the eigenvalues of $\Ascr$ are observable.} \hfill $\square$ \end{pf}

{
We mention that the condition \rfb{Hegseth} is numerically
unverifiable, because the set of all the numbers of the type
$(2n-1)/(2m-1)$ is dense in $[0,\infty)$.}


{
To solve the output regulator problem for the piezoelectric beam, we
interconnect the system $\Sig_p$ with the \textbf{IMC} controller
\rfb{minreal} with input $e=r-y$ and output $u^{out}$, as shown in
Fig.~\ref{fig:blocks}, so that $u=v+u^{out}$. The resonant frequencies
of the \textbf{IMC} are, of course, the frequencies $\o_j$ from
\rfb{exosig}. }

\begin{theorem} {
Suppose $i\o_j\in\rho(\Ascr)$ and $\PPP_u(i\o_j)\neq 0$ for all $j\in
\Jscr$. Then the plant $\Sig_p$ from \rfb{eq:piezobeam},
interconnected with the controller $\Sig_c$ from \rfb{minreal} as in
Fig.~\ref{fig:blocks}, is well-posed, input-output stable and
impedance passive from $\sbm{w\\ v\\ r}$ to $\sbm{y\\ u^{out}}$, with
the state space $X\times\cline^\ell$, where $X$ is given by
\rfb{statespace2} and $\cline^\ell$ is the state space of \m
$\Sig_c$.}

{
If $\frac{\zeta_1}{\zeta_2}\neq\frac{2n-1}{2m-1}$ for all $n,m\in
\nline$, then this controller solves the output regulator problem as
stated in Section \ref{sec1}, and moreover $\|e(t)\|^2$ tends to zero
when filtered.}
\end{theorem} \vspace{-2mm}

\begin{pf} {
Clearly Assumption \ref{ass_trans} and also all the other requirements
in Theorem \ref{thm:main} are satisfied. From the first part of
Theorem \ref{thm:main} we obtain that the closed-loop system is
well-posed, input-output stable and impedance passive from $\sbm{w\\
v\\ r}$ to $\sbm{y\\ u^{out}}$.}

{
Clearly $\Ker\m B=\{0\}$. From Proposition \ref{prop:exa2_skew},
$\sigma(\Ascr)$ is a countable set of eigenvalues on $i\rline$. By
Proposition \ref{prop:obs2}, if the condition about $\zeta_1/\zeta_2$
holds, then there are no unobservable eigenvalues of $\Ascr$.
According to Theorem \ref{thm:main}, $\Sig_c$ solves the output
regulator problem for $\Sig_p$ and $\|e(t)\|^2$ tends to zero when
filtered.} \hfill $\square$ \m \end{pf}

\begin{figure}[htbp] \centering 
   \includegraphics[width=0.7\linewidth]{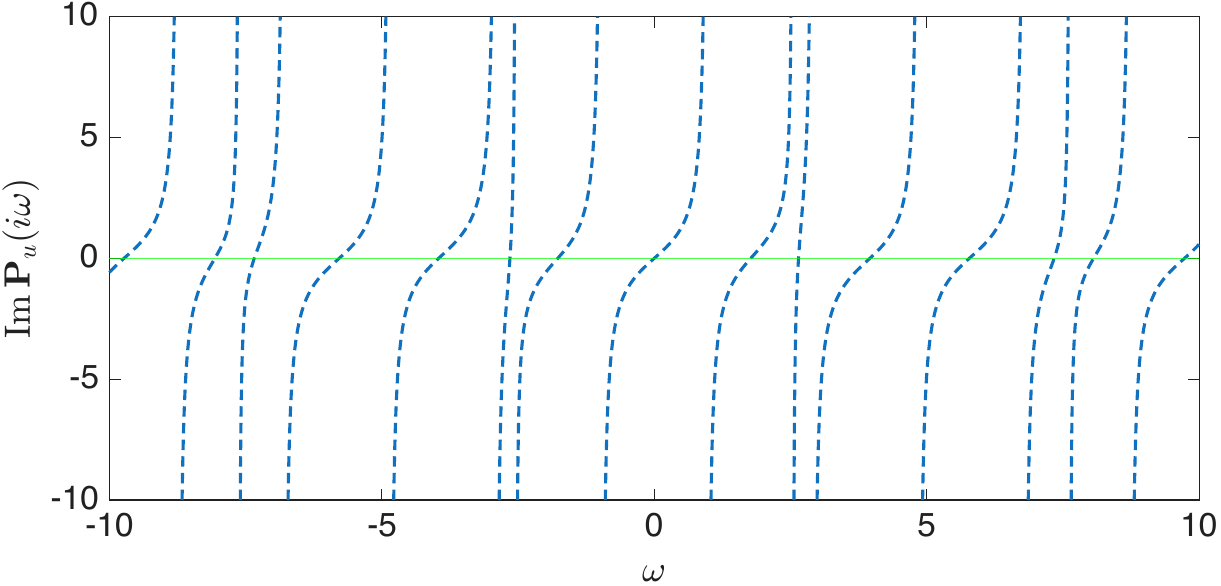}
   \caption{The function $\Im{\PPP_u}(i\o)$. } \label{fig1:IM2}
\end{figure}

\begin{figure}[htbp] \centering 
  \begin{subfigure}[t]{0.235\textwidth}
     \includegraphics[width=\linewidth]{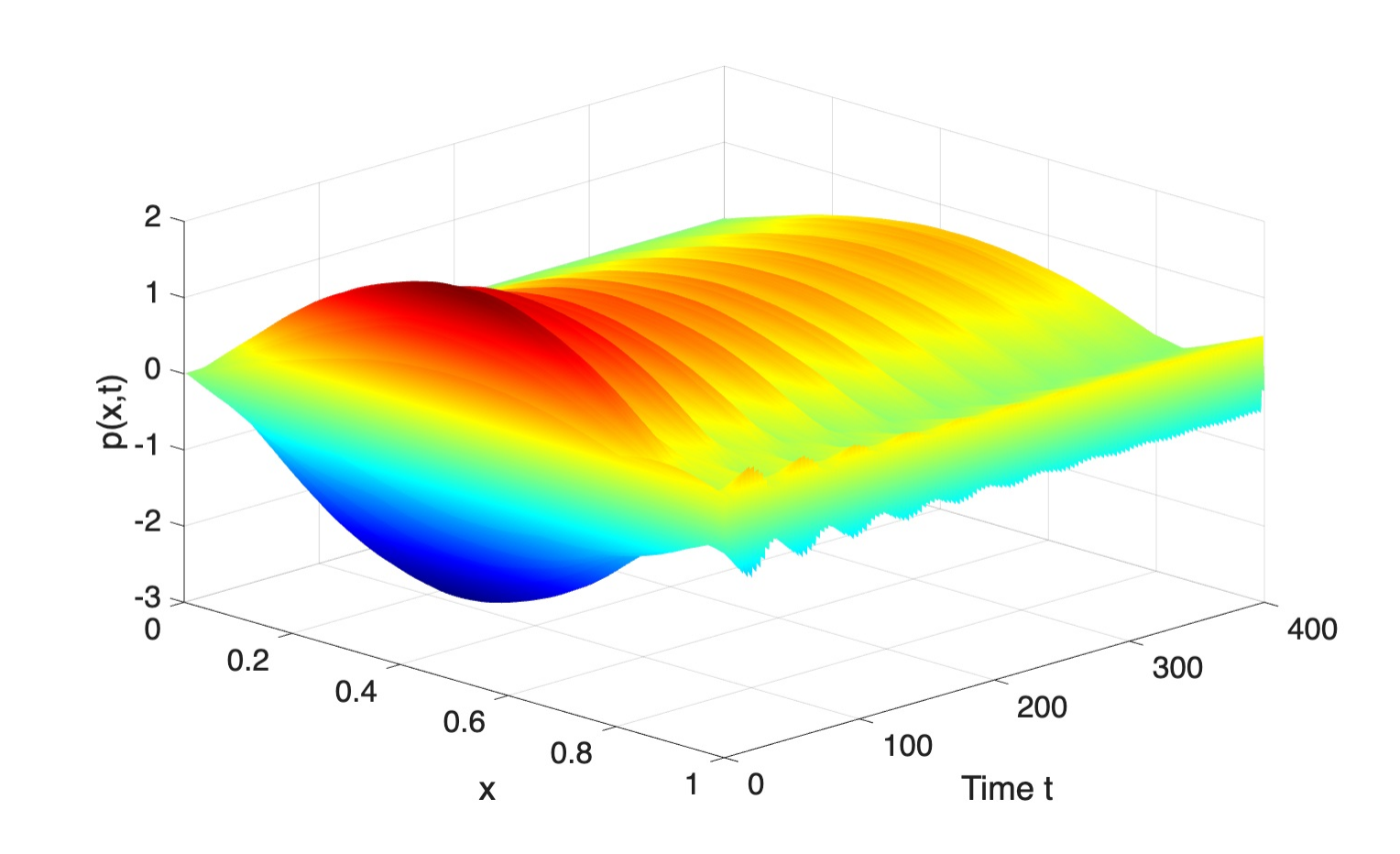} 
     \caption{Plot of $p(x,t)$.} \label{fig:sub1}
     \end{subfigure} \hfill
  \begin{subfigure}[t]{0.235\textwidth}
     \includegraphics[width=\linewidth]{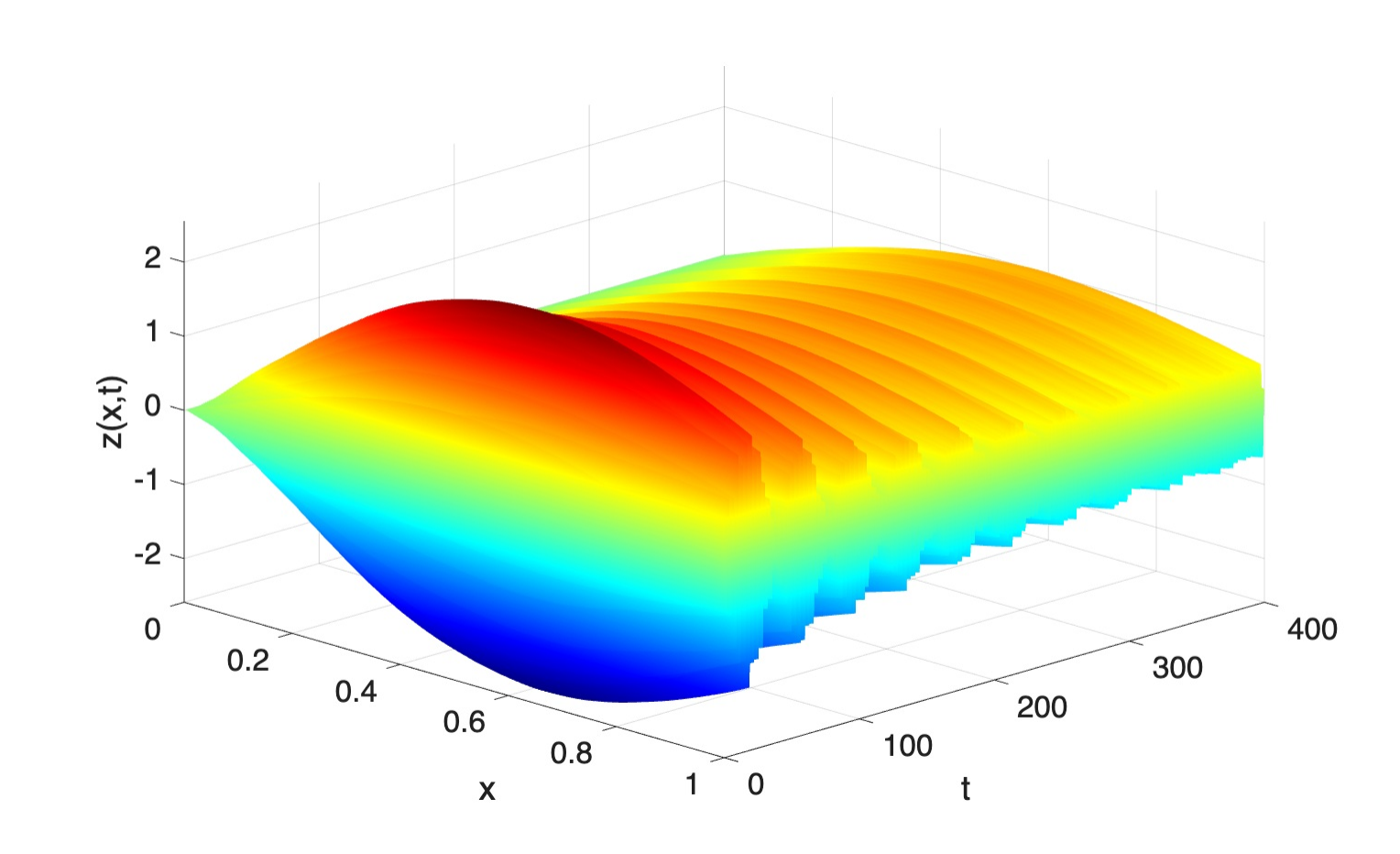} 
     \caption{Plot of $z(x,t)$.} \label{fig:sub2}
     \end{subfigure}
  \caption{System response of \dref{eq:piezobeam} with $\Sig_c$.}
  \label{fig:combined}
\end{figure}

\vspace{-2mm} {
{\bf Simulation setup:} \m In the closed-loop system of Fig.
\ref{fig:blocks}, we consider the disturbances $w_1(t)=w_2(t)=0$,
$v(t)=0.2\cos(2t+1)$, and the reference signal $r(t)=\sin(2t)$. The
system parameters are $\rho=\gamma=\beta=\mu=h=1$ and $\alpha=2$. The
parameters of the controller are selected as $c_1=1$, $\o_1=2$, and
$d=1$, and its impedance passive and minimal realization $\Sig_c$ is
\begin{align*}
   \dot{q}(t) &=\m \bbm{0 & -2 \\ 2 & 0} q(t)+\bbm{1 \\ 1} e(t),\\
   u^{out}(t) &=\m \bbm{1 & 1} q(t) + e(t).
\end{align*}}

{\bf Results:} A numerical computation indicates that $\Re\PPP_u=0$ on
$i\rline\cap \rho(A)$, but this is not important, what matters is that
$|\Im\m \PPP_u(\pm 2i)|\approx 0.6$, so that $\PPP_u(\pm 2i)$ is
invertible. Fig.~\ref{fig:combined} shows the state evolution of the
system governed by \rfb{eq:piezobeam} with the \textbf{IMC}
controller. Fig.~\ref{fig:combined2} further validates the
controller's effectiveness by illustrating the tracking performance
under the specified disturbances and reference signal.

\begin{figure}[htbp] \centering 
   \begin{subfigure}[t]{0.235\textwidth}
   \includegraphics[width=\linewidth]{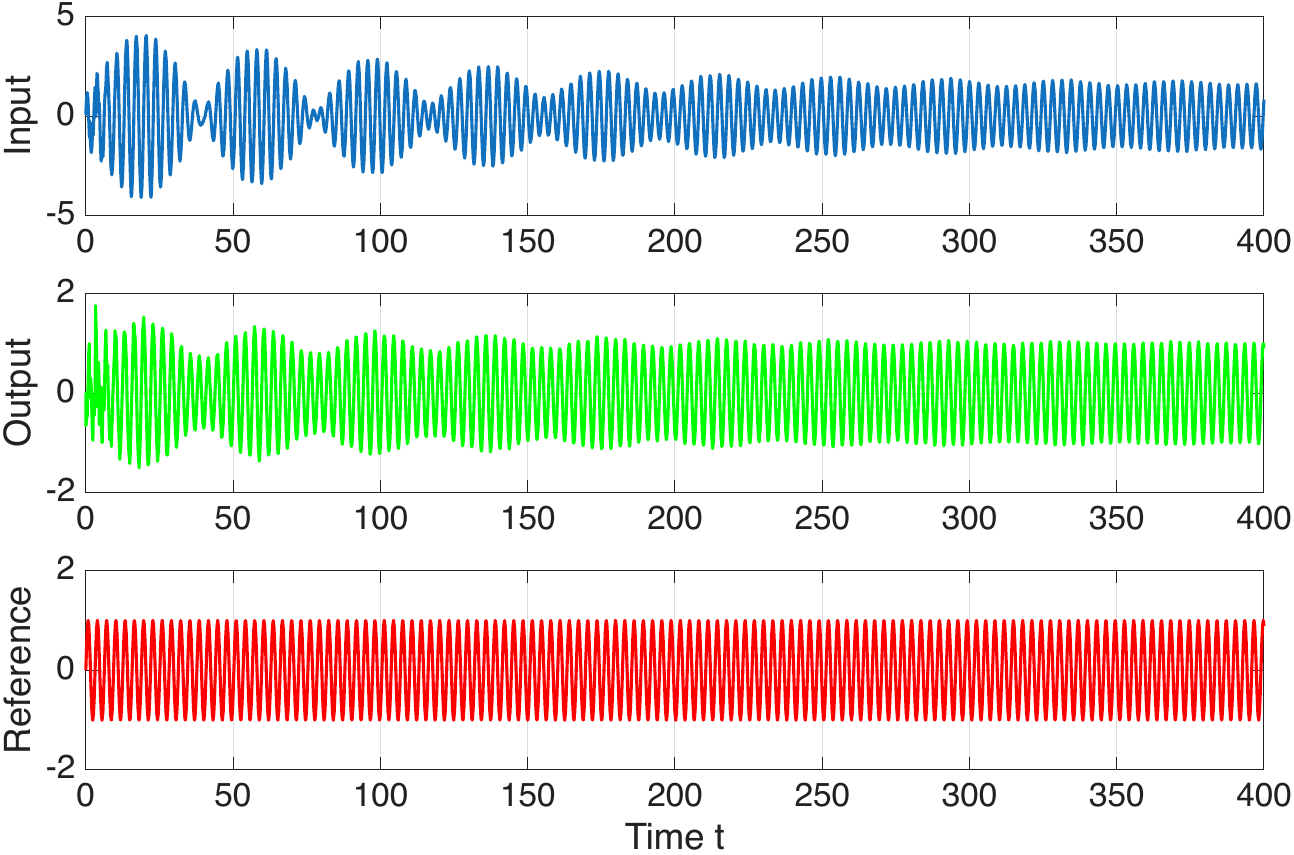}
   \caption{The input $u$, output $y$ and reference $r$.}
   \label{fig:sub3} \end{subfigure}
   \begin{subfigure}[t]{0.235\textwidth}
   \includegraphics[width=\linewidth]{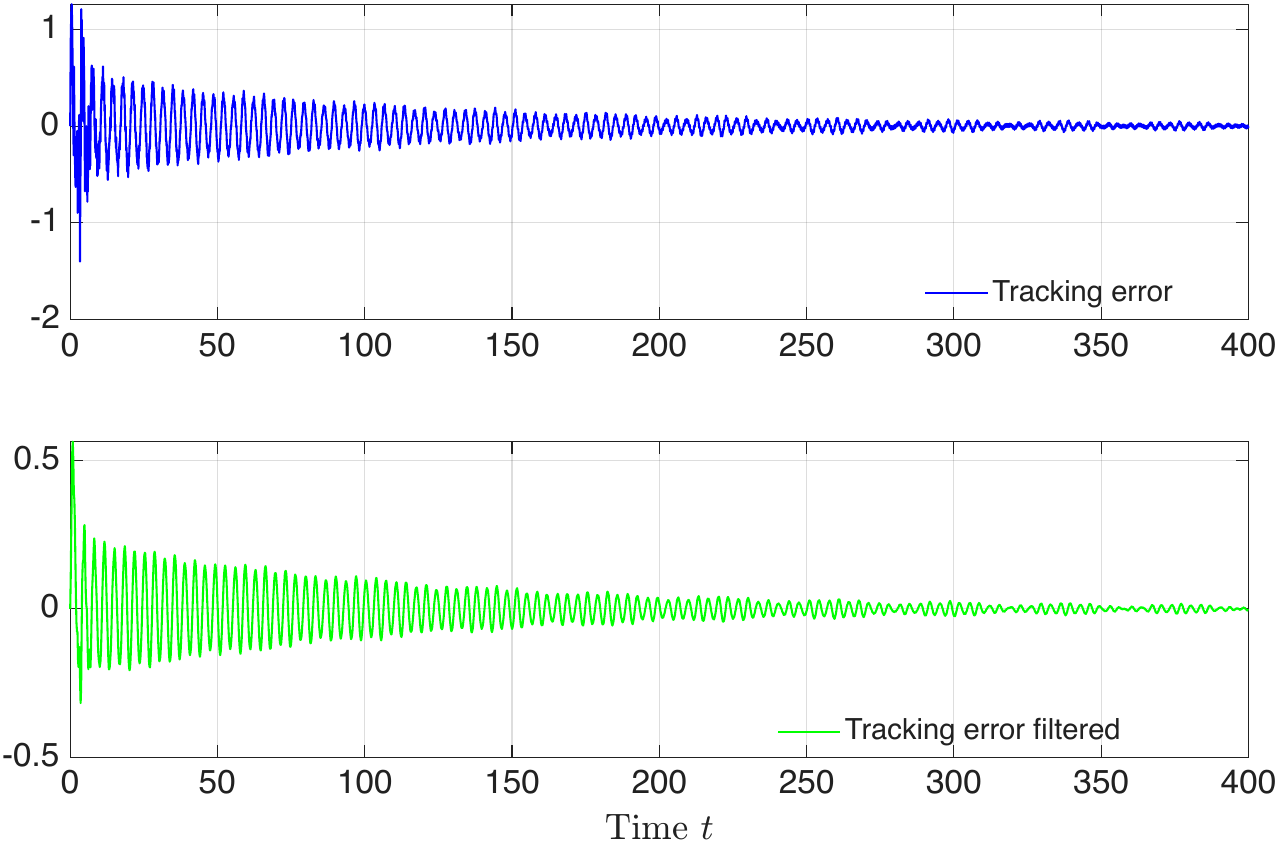}
   \caption{The tracking error $e$.} \label{fig:sub4} \end{subfigure}
   \caption{Tracking performance of \dref{eq:piezobeam} with
   $\Sig_c$.} \label{fig:combined2} \vspace{-3mm}
\end{figure}

\section{Appendix} \setcounter{equation}{0} 

Here we prove three facts stated in Section \ref{sec2}.

\noindent
{\bf Fact I}: The function 
$v(t)=\sin (t^2)$ is not in $L^2[0,\infty)$, and $v$ tends to zero
when filtered.

{\bf Proof.} \m From
$$ \begin{array}{l} \int_0^\infty |\sin(s^2)|^2\dd s \m=\m \int_0
   ^\infty \frac{|\sin\tau|^2}{2\sqrt{\tau}}\dd\tau\\ =\m \sum_{k=0}
   ^\infty \int_0^\pi \frac{\sin^2\tau}{2\sqrt{\tau+k\pi}} \dd\tau
   \m\geq\m \sum_{k=0}^\infty \int_0^\pi \frac{\sin^2\tau}{2
   \sqrt{\pi+k\pi}}\dd\tau\\ =\m \sum_{k=0}^\infty \frac{\pi}{4
   \sqrt{\pi+k\pi}} \m=\m \infty \end{array}$$
we obtain that $v\notin L^2[0,\infty)$.

Next, we show that $v$ tends to zero when filtered. Defining $t_n=
\sqrt{2n\pi+3\pi}$, where $n\in\nline$, we calculate
$$ \scalebox{0.9}{$\displaystyle
   \int_{t_0}^{t_n} e^s \sin(s^2) \dd s \m=\m \sum_{k=1}^n
   \int_{t_{k-1}}^{t_k} e^s \sin(s^2) \dd s$}$$
\begin{equation} \label{int-sint2-estm}
   \scalebox{0.9}{$\displaystyle \begin{aligned}
   & = \frac{1}{2} \sum_{k=1}^n \biggl[
   \int_{2k\pi+\pi}^{2k\pi+2\pi} \frac{e^{\sqrt{\tau}}}{\sqrt{\tau}}
   \sin(\tau) d\tau + \int_{2k\pi+2\pi}^{2k\pi+3\pi}
   \frac{e^{\sqrt{\tau}}}{\sqrt{\tau}}\sin(\tau) d\tau \biggr] \\
   & = \frac{1}{2}\sum_{k=1}^n \biggl[\int_{2k\pi + \pi}^{2k\pi+2\pi}
   \frac{e^{\sqrt{\tau}} \sin(\tau)}{\sqrt{\tau}} - \frac{e^{\sqrt{
   \tau+\pi}}\sin(\tau)}{\sqrt{\tau+\pi}}\dd\tau \biggr] \\
   & = \frac{1}{2} \sum_{k=1}^{n} \int_{2k\pi+\pi}^{2k\pi+2\pi}
   [h(\tau+\pi)-h(\tau)](-\sin(\tau)) \dd\tau \end{aligned}$}
\end{equation}
where $h(s)=\frac{e^{\sqrt{s}}}{\sqrt{s}}$, $s\geq\pi$. Since $h(s)$
is strictly increasing function on $s\in[\pi,+\infty)$ due to
$h'(s)=({2s})^{-1}{e^{\sqrt{s}}(1-s^{-1/2})}>0$ for $s\geq\pi$, we
obtain for $\tau\in[2k\pi+\pi,2k\pi+2\pi]$,
$$ 0 \m<\m h(\tau+\pi)-h(\tau)\leq h(2k\pi+3\pi)-h(2k\pi+\pi),$$
which, jointly with \dref{int-sint2-estm}, gives
\begin{equation*}
   \begin{array}{l} 0 \m<\m \int_{t_0}^{t_n} e^s\sin(s^2)\dd s\\
   \ \ \ \leq\m \half\sum_{k=1}^n [h(2k\pi+3\pi)-
   h(2k\pi+\pi)]\times \\
   \qquad \int_{2k\pi+\pi}^{2k\pi+2\pi} (-\sin(\tau))\dd\tau \\
   \ \ \ =\m h(2n\pi+3\pi)-h(3\pi). \end{array}
\end{equation*}
Combining these facts, it follows that
\begin{equation} \label{int-sint2-estm-ii-}
   \begin{array}{l} 0 \m\leq\m \lim_{n\to+\infty}\frac{\int_{t_0}^{t_n}
   e^s \sin(s^2)\dd s}{e^{t_n}}\\ \ \ \leq\m \lim_{n\to+\infty}
   \frac{h(2n\pi+3\pi)-h(3\pi)}{e^{t_n}} \m=\m 0. \end{array}
\end{equation}
For sufficient large $t$, there exists a $n\in\nline$ such that
$t\in[t_n,t_{n+1}]$. From \dref{int-sint2-estm-ii-} and $0<
\frac{e^{t_n}}{e^t}\leq1$, we have
\begin{equation} \label{int-sint2-estm-ii-fn}
   \begin{array}{l} \lim_{t\to+\infty}\int_0^t e^{-(t-s)}\sin(s^2)\dd s\\
   =\m \lim_{t\to+\infty}\Big[\frac{\int_0^{t_1} e^s\sin(s^2)\dd s}{e^{t}}
   + \frac{\int_{t_0}^{t_n} e^{s}\sin(s^2)\dd s}{e^{t_n}}\frac{e^{t_n}}
   {e^t}\\ \hspace{1.98cm} + \frac{\int_{t_n}^t e^s\sin(s^2)\dd s}{e^t}
   \Big]\\ =\m \lim_{t\to+\infty} \frac{\int_{t_n}^t e^s \sin(s^2)\dd s}
   {e^t}\\ =\m \lim_{t\to+\infty} \frac{e^{\xi_t}\sin(\xi_t^2)(t-t_n)}
   {e^t} \m=\m 0. \end{array}
\end{equation}
In the last step of \dref{int-sint2-estm-ii-fn}, we have used that
$$ 0 \leq t-t_n \leq t_{n+1}-t_n = \frac{2\pi}{t_{n+1}+t_n}\to 0
   \mbox{ as } n\to+\infty$$
and the mean value theorem to guarantee that $\xi_t\in(t_n,t)$ and
$0<\frac{e^{\xi_t}}{e^t}\leq 1$. \hfill $\square$ \m

\noindent
{\bf Fact II}: $v(t)=|\sin t|^{t^3}$ does not tend to zero but $v\in
L^2[0,\infty)$ and thus $v$ tends to zero when filtered.

\begin{pf} It is clear that $v(t)=|\sin t|^{t^3}$ does not tend to
zero. In order to show that $v$ tends to zero when filtered, by
Proposition \ref{Hodeida}, it suffices to show that $v\in L^2[0,
\infty)$. Note that
\begin{equation} \label{app-estim-sint-t3-in}
  \begin{array}{l} \int_{0}^{+\infty}|\sin s|^{2s^3}\dd s \m=\m
    \sum_{j=0}^\infty \int_{j\pi}^{(j+1)\pi} |\sin s|^{2s^3}\dd s\\ =\m
    \sum_{j=0}^\infty \int_0^\pi |\sin \tau|^{2(\tau+j\pi)^3}\dd\tau\\
    \leq\m \int_0^\pi |\sin \tau|^{2\tau^3}\dd\tau + \sum_{j=1}^\infty
    \int_0^\pi |\sin \tau|^{(2j)^3}\dd\tau. \end{array}
\end{equation}
By the well-known Wallis formula, see for instance
\cite{Miller-pi2008},
\begin{equation} \label{Wallis-Eq1}
  \int_0^{\pi/2}\sin^{2n}x\dd x \m=\m \frac{(2n-1)!!}{(2n)!!}
  \frac{\pi}{2},
\end{equation}
where $m!!=m(m-2)(m-4)...$ \ (the product stops when either 1 or 2
is reached) and using the Wallis limit (see the same reference)
\begin{equation} \label{Wallis-Eq2}
  \lim_{n\to+\infty} \Big[ \frac{(2n)!!}{(2n-1)!!}\Big]^2\frac{1}
  {2n+1} \m=\m \frac{\pi}{2},
\end{equation}
we derive that
$$ \begin{array}{l} \int_0^\pi |\sin \tau|^{(2j)^3} \dd\tau \m=\m
   \frac{((2j)^3-1)!!}{((2j)^3)!!}\frac{\pi}{2}\\ \hspace{2.7cm}
   \thicksim\frac{1}{\sqrt{(2j)^3+1}}\sqrt{\frac{\pi}{2}},\ \
   \mbox{as }j\to+\infty. \end{array}$$
Noting that $ \sum_{j=1}^\infty \frac{1}{\sqrt{(2j)^3+1}} \m<\m
\infty,$ by \dref{app-estim-sint-t3-in}, we obtain that \
$\int_0^\infty|\sin s|^{2s^3}\dd s < \infty$, i.e., $v\in L^2[0,
\infty)$. \hfill $\square$ \end{pf}

\noindent
{\bf Fact III}: The function $v:[0,\infty)\rarrow\rline$ defined by
$v(t)=|\sin t|^t$ is not in $L^2[0,\infty)$, but $v\rarrow 0$ when
filtered.

\begin{pf} \m We first show that $v\notin L^2[0,\infty)$. We see
that
\begin{equation} \label{app-estim-sint-tt-sta}
   \begin{array}{l} \int_0^{+\infty} |\sin s|^{2s}\dd s \m=\m
   \sum_{j=0}^\infty \int_{j\pi}^{(j+1)\pi} |\sin s|^{2s}\dd s\\
   =\m \sum_{j=0}^\infty \int_0^\pi |\sin \tau|^{2(\tau+j\pi)}\dd
   \tau \m\geq\m \sum_{j=1}^\infty \int_0^\pi |\sin \tau|^{8j}\dd
   \tau. \end{array}
\end{equation}
By \dref{Wallis-Eq1} and \dref{Wallis-Eq2}, we have
\begin{equation} \label{app-estim-sint-tt-sta-1}
   \begin{array}{l} \int_0^\pi |\sin \tau|^{8j}\dd\tau \m=\m
   \frac{(8j-1)!!}{(8j)!!}\frac{\pi}{2}\\ \hspace{2.3cm}\thicksim
   \frac{1}{\sqrt{8j+1}}\sqrt{\frac{\pi}{2}},\ \ \mbox{as }\
   j\to\infty. \end{array}
\end{equation}
Since
$\sum_{j=1}^\infty \frac{1}{\sqrt{8j+1}} \m= \infty \m,$
by \rfb{app-estim-sint-tt-sta} and \dref{app-estim-sint-tt-sta-1},
we obtain that $\int_0^\infty |\sin s|^{2s}\dd s=\infty$, so that
$v\notin L^2[0,\infty)$.

Let $y(t)=\frac{1}{\pi}\int_{t-\pi}^t v(s)\dd s$. It is easy to see
that $\displaystyle \lim_{t\to+\infty} y(t)=0$. By Prop.
\ref{Fordo}, $v\rarrow 0$ when filtered. \hfill $\square$ \m
\end{pf}

\bibliographystyle{plain}


\parpic{\includegraphics[width=25mm,height=32mm,clip,
  keepaspectratio]{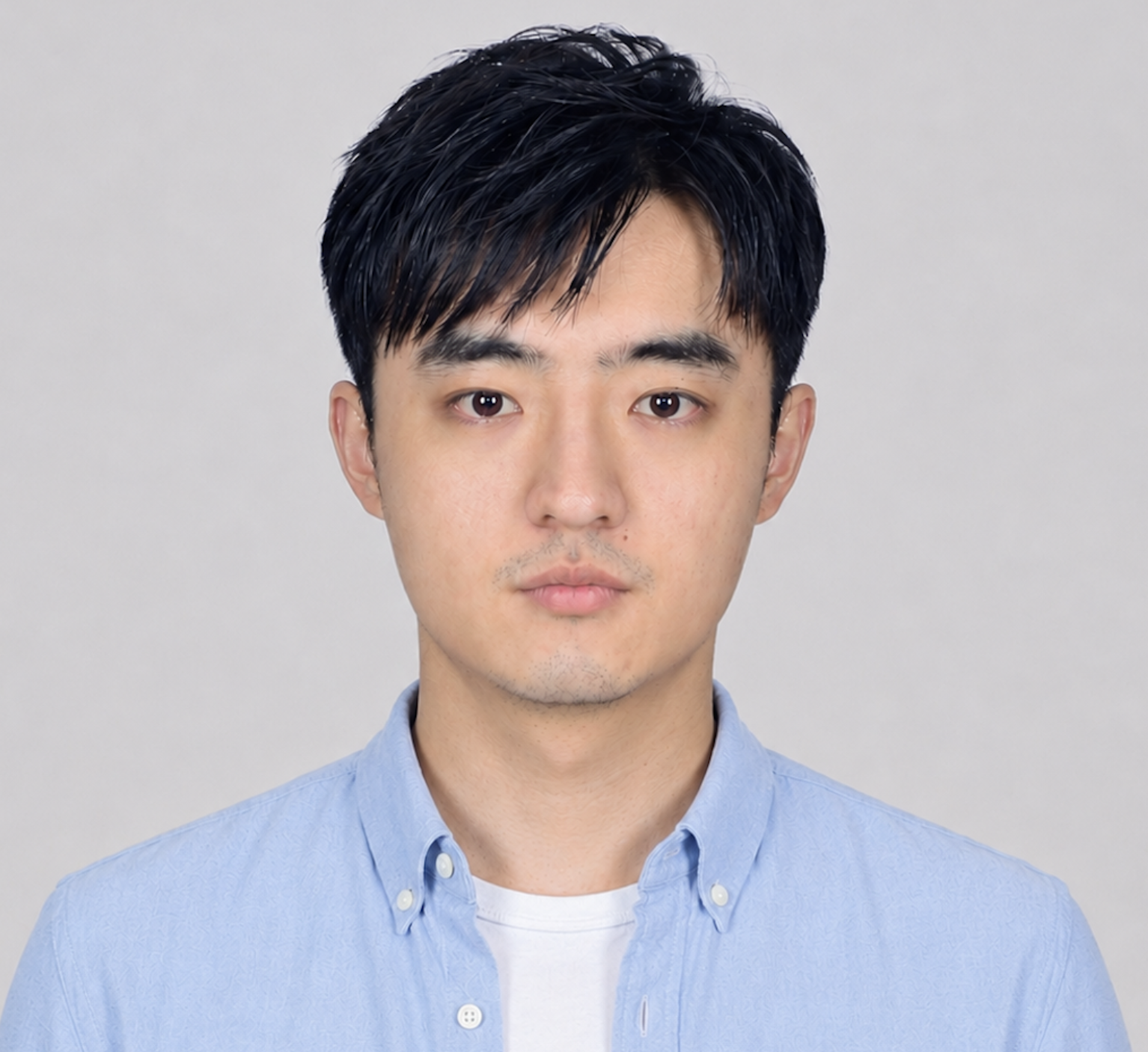}}
\noindent \textbf{Bingsen Li} received the PhD degree from the
  School of Mathematics and Statistics, Central South University,
  China in 2026. He is also currently a PhD student with the
  School of Electrical and Computer Engineering, Tel Aviv
  University, Israel. His current research interests include
  distributed parameter systems, output regulation, and sampled
  data system.

\parpic{\includegraphics[width=25mm,height=32mm,clip,
  keepaspectratio]{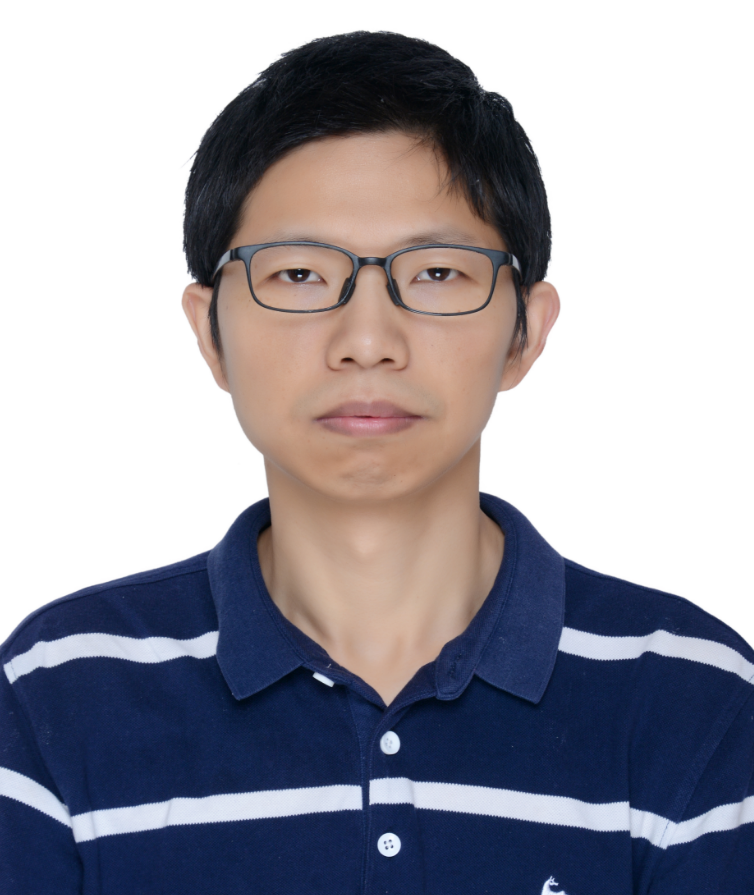}}
\noindent \textbf{Hua-Cheng Zhou} received his PhD degree in
  applied mathematics from the Academy of Mathematics and System
  Sciences, Chinese Academy of Sciences, Beijing, China in 2015.
  From 2015 to 2018, he was a Postdoctoral Fellow at the School
  of Electrical Engineering, Tel Aviv University, Tel Aviv,
  Israel. Since 2018, he has been with the School of Mathematics
  and Statistics, Central South University, where he is a Full
  Professor in Applied Mathematics. His current research interests
  focus on distributed parameter systems control.

\noindent \textbf{George Weiss} is a Professor in the School of
  Electrical and Computer Engineering at Tel Aviv University. He
  received the PhD degree in applied mathematics from The Weizmann
  Institute, Rehovot, Israel, in 1989. He was with Brown
  Univ., Providence, RI, Virginia Tech, Blacksburg, VA, Ben-Gurion
  Univ., Beer Sheva, Israel, the Univ. of Exeter, UK, and Imperial
  College London, UK. His current research interests include
  distributed parameter systems, operator semigroups, passive
  systems (linear and nonlinear), power electronics, and the grid
  integration of distributed energy sources.

\end{document}